\documentclass[12pt]{amsart}

\usepackage{amsmath}
\usepackage{amssymb}
\usepackage{ascmac}
\usepackage{amsthm}

\makeatletter

\@addtoreset{equation}{section}
\makeatother 

\theoremstyle{plain}
\newtheorem{thm}{Theorem}[section]
\newtheorem*{thm*}{Theorem}
\newtheorem{prop}[thm]{Proposition}
\newtheorem{lem}[thm]{Lemma}
\newtheorem{cor}[thm]{Corollary}

\theoremstyle{definition}

\theoremstyle{remark}
\newtheorem{rem}{Remark}[section]

\newcommand{\vol}{\operatorname{vol}}
\newcommand{\ric}{\operatorname{Ric}}
\newcommand{\Div}{\operatorname{div}}
\newcommand{\Hess}{\operatorname{Hess}}

\newcommand{\tr}{\operatorname{trace}}
\newcommand{\IR}{\operatorname{InRad}}

\newcommand{\cut}{\mathrm{Cut}\,}
\newcommand{\inte}{\mathrm{Int}\,}

\newcommand{\supp}{\mathrm{supp}\,}

\newcommand{\bm}{\partial M}
\newcommand{\ball}{B^{n}_{\kappa,\lambda}}
\newcommand{\const}{C_{\kappa,\lambda}}

\newcommand{\bball}{\partial B^{n}_{\kappa,\lambda}}

\usepackage{latexsym}

\title[Comparison geometry of manifolds with boundary]{Comparison geometry of manifolds\\ with boundary under a lower weighted\\ Ricci curvature bound}
\author{Yohei Sakurai}
\date{December 7, 2017}
\address{Institute for Applied Mathematics, University of Bonn, Endenicher Allee 60, D-53115 Bonn, Germany}
\email{sakurai@iam.uni-bonn.de}
\thanks{Research Fellow of Japan Society for the Promotion of Science for 2014-2016}
\subjclass[2010]{53C20}
\keywords{Manifold with boundary; Weighted Ricci curvature}
\begin{document}
\maketitle

\begin{abstract}
We study Riemannian manifolds with boundary under a lower weighted Ricci curvature bound.
We consider a curvature condition in which
the weighted Ricci curvature is bounded from below by the density function.
Under the curvature condition,
and a suitable condition for the weighted mean curvature for the boundary,
we obtain various comparison geometric results.
\end{abstract}

\section{Introduction}\label{sec:Introduction}
We study comparison geometry of manifolds with boundary under a lower weighted Ricci curvature bound.
For the lower weighted Ricci curvature bound,
we consider a curvature condition
in which the lower bound is controlled by the density function.
We introduce a reasonable curvature condition for a lower weighted mean curvature bound for the boundary.
Under these curvature conditions,
we investigate comparison geometric properties,
and conclude twisted rigidity theorems.

For $n\geq 2$,
let $(M,g)$ be an $n$-dimensional Riemannian manifold with or without boundary,
and let $f:M\to \mathbb{R}$ be a smooth function.
Let $\ric_{g}$ denote the Ricci curvature defined by $g$.
For $N\in (-\infty,\infty]$,
the \textit{$N$-weighted Ricci curvature} is defined as follows:
If $N \in (-\infty,\infty)\setminus \{n\}$,
\begin{equation}\label{eq:def of weighted Ricci curvature}
\ric^{N}_{f}:=\ric_{g}+\Hess f-\frac{d f \otimes d f}{N-n},
\end{equation}
where $d f$ and $\Hess f$ are the differential and the Hessian of $f$,
respectively;
otherwise,
if $N=\infty$,
then $\ric^{N}_{f}:=\ric_{g}+\Hess f$;
if $N=n$,
and if $f$ is constant,
then $\ric^{N}_{f}:=\ric_{g}$;
if $N=n$,
and if $f$ is not constant,
then $\ric^{N}_{f}:=-\infty$ (\cite{BE}, \cite{Lic}).
For a smooth function $K:M\to \mathbb{R}$,
we mean by $\ric^{N}_{f,M}\geq K$
for every point $x \in M$,
and every unit tangent vector $v$ at $x$,
it holds that $\ric^{N}_{f}(v)\geq K(x)$.

For manifolds without boundary whose $N$-weighted Ricci curvatures are bounded from below by constants,
many comparison geometric results have been already known in the usual weighted case of $N\in [n,\infty]$ (see e.g., \cite{Lic}, \cite{Lo}, \cite{Q}, \cite{WW}).
For manifolds with boundary,
the author \cite{Sa2} has studied such comparison geometric properties.

Recently,
under a lower $N$-weighted Ricci curvature bound,
Wylie \cite{W}, and Wylie and Yeroshkin \cite{WY} have studied comparison geometry of manifolds without boundary in complementary case of $N \in (-\infty,n)$.
Wylie \cite{W} has obtained a splitting theorem of Cheeger-Gromoll type under the curvature condition $\ric^{N}_{f,M}\geq 0$ for $N \in (-\infty,1]$.
Wylie and Yeroshkin \cite{WY} have introduced a curvature condition
\begin{equation*}
\ric^{1}_{f,M}\geq (n-1)\, \kappa\, e^{\frac{-4f}{n-1}}
\end{equation*}
for $\kappa \in \mathbb{R}$ from the view point of study of weighted affine connections.
Under such condition,
they have proved a maximal diameter theorem of Cheng type for the distance induced from the metric $e^{\frac{-4f}{n-1}}g$,
and a volume comparison of Bishop-Gromov type for the measure $e^{-\frac{n+1}{n-1}f} \vol_{g}$,
where $\vol_{g}$ denotes the Riemannian volume measure on $(M,g)$.

In this paper,
we study comparison geometry of Riemannian manifolds with boundary satisfying the curvature condition
\begin{equation}\label{eq:Ricci curvature assumption}
\ric^{N}_{f,M}\geq (n-1)\, \kappa\, e^{\frac{-4f}{n-1}}
\end{equation}
for $\kappa \in \mathbb{R},\,N \in (-\infty,\infty]$.
We will also consider a curvature condition for the boundary
that is compatible with (\ref{eq:Ricci curvature assumption}).
For a Riemannian manifold $M$ with boundary,
let $\bm$ stand for its boundary.
For $z\in \bm$,
we denote by $u_{z}$ the unit inner normal vector on $\bm$ at $z$,
and by $H_{z}$ the mean curvature of $\bm$ at $z$ with respect to $u_{z}$ (more precisely, see Subsection \ref{sec:Jacobi fields orthogonal to the boundary}).
The \textit{weighted mean curvature} $H_{f,z}$ is defined as
\begin{equation*}
H_{f,z}:=H_{z}+g\left((\nabla f)_{z},u_{z}\right),
\end{equation*}
where $\nabla f$ is the gradient of $f$.
We introduce a curvature condition
\begin{equation}\label{eq:mean curvature assumption}
H_{f,\bm}\geq (n-1)\, \lambda\, e^{\frac{-2f}{n-1}}
\end{equation}
for $\lambda \in \mathbb{R}$,
where (\ref{eq:mean curvature assumption}) means that
$H_{f,z}\geq (n-1)\lambda e^{\frac{-2f(z)}{n-1}}$ for every $z\in \bm$.
Under the conditions (\ref{eq:Ricci curvature assumption}) and (\ref{eq:mean curvature assumption}) for $\kappa,\lambda \in \mathbb{R},\,N \in (-\infty,1]$,
we formulate various comparison geometric results,
and generalize the preceding studies by Kasue \cite{K2}, \cite{K3},
and the author \cite{Sa1} when $f=0$.

\subsection{Setting}\label{sec:setting}
In the present paper,
we work in the following setting:
For $n\geq 2$,
let $(M,g)$ be an $n$-dimensional, 
connected complete Riemannian manifold with boundary,
and let $f:M\to \mathbb{R}$ be a smooth function.
For $\kappa,\lambda \in \mathbb{R},\,N \in (-\infty,\infty]$
we say that
a triple $(M,\bm,f)$ has \textit{lower $(\kappa,\lambda,N)$-weighted curvature bounds}
if (\ref{eq:Ricci curvature assumption}) and (\ref{eq:mean curvature assumption}) hold.
For $N_{1} \in (n,\infty],\,N_{2} \in (-\infty,n)$,
or for $N_{1},\,N_{2} \in (-\infty,n)$ with $N_{1}\leq N_{2}$,
if $(M,\bm,f)$ has lower $(\kappa,\lambda,N_{1})$-weighted curvature bounds,
then it also has lower $(\kappa,\lambda,N_{2})$-weighted curvature bounds (see (\ref{eq:def of weighted Ricci curvature}) and (\ref{eq:Ricci curvature assumption})).
We mainly study a triple $(M,\bm,f)$ with lower $(\kappa,\lambda,N)$-weighted curvature bounds for $\kappa,\lambda \in \mathbb{R},\,N \in (-\infty,1]$.

\subsection{Splitting theorems}
For the Riemannian distance $d_{M}$ on $M$,
let $\rho_{\bm}:M\to \mathbb{R}$ stand for the distance function from the boundary $\bm$ defined as $\rho_{\bm}(x):=d_{M}(x,\bm)$.
For $z\in \bm$,
let $\gamma_{z}:[0,T)\to M$ be the geodesic with $\gamma_{z}'(0)=u_{z}$.
Define functions $\tau,\,\tau_{f}:\bm \to (0,\infty]$ by
\begin{equation}\label{eq:boundary cut value}
\tau(z):=\sup \{\,t>0 \mid \rho_{\bm}(\gamma_{z}(t))=t\,\},\quad \tau_{f}(z):=\int^{\tau(z)}_{0}\,e^{\frac{-2f(\gamma_{z}(a))}{n-1}}\,da.
\end{equation}
We also define a function $s_{f,z}:[0,\tau(z)]\to [0,\tau_{f}(z)]$ by
\begin{equation}\label{eq:speed changing function}
s_{f,z}(t):=\int^{t}_{0}\,  e^{\frac{-2f(\gamma_{z}(a))}{n-1}}\,da.
\end{equation}
Let $s_{\kappa,\lambda}(s)$ be a unique solution of the Jacobi equation $\varphi''(s)+\kappa \varphi(s)=0$ with $\varphi(0)=1,\,\varphi'(0)=-\lambda$.
For $t \in [0,\tau(z)]$
we set
\begin{equation}\label{eq:model Jacobi field}
F_{\kappa,\lambda,z}(t):=  \exp \left(  \frac{f(\gamma_{z}(t))-f(z)}{n-1}  \right)\,s_{\kappa,\lambda}(s_{f,z}(t)).
\end{equation}
Let $h$ denote the induced Riemannian metric on $\bm$.
For an interval $I$,
and a connected component $\bm_{1}$ of $\bm$,
we denote by $I \times_{F_{\kappa,\lambda}} \bm_{1}$ the twisted product Riemannian manifold $\bigl(I \times \bm_{1},dt^{2}+F^{2}_{\kappa,\lambda,z}(t)\,h\bigl)$.

One of the main results is the following twisted splitting theorem:
\begin{thm}\label{thm:splitting theorem}
Let $\kappa \leq 0$ and $\lambda:=\sqrt{\vert \kappa \vert}$.
For $N \in (-\infty,1]$,
assume that
$(M,\bm,f)$ has lower $(\kappa,\lambda,N)$-weighted curvature bounds.
Suppose that
$f$ is bounded from above.
If $\tau(z_{0})=\infty$ for some $z_{0}\in \bm$,
then $M$ is isometric to $[0,\infty)\times_{F_{\kappa,\lambda}} \bm$;
moreover,
if $N \in (-\infty,1)$,
then for every $z\in \bm$
the function $f\circ \gamma_{z}$ is constant on $[0,\infty)$.
\end{thm}
When $\kappa=0$ and $\lambda=0$,
Theorem \ref{thm:splitting theorem} was proven by the author in the cases $N \in [n,\infty]$ (see \cite{Sa2}) and $N \in (-\infty,1]$ (see \cite{Sa3}).
In the unweighted case of $f=0$,
Kasue \cite{K2} has proved Theorem \ref{thm:splitting theorem}
under the assumption that
$M$ is non-compact and $\bm$ is compact (see also Croke and Kleiner \cite{CK}),
and the author \cite{Sa1} has proved Theorem \ref{thm:splitting theorem} itself.

In Theorem \ref{thm:splitting theorem},
by applying a splitting theorem proved by Wylie \cite{W} to the boundary,
we obtain a multi-splitting theorem (see Subsection \ref{sec:Multi-splitting}).
We also generalize a splitting theorem studied by Kasue \cite{K2} (and Croke and Kleiner \cite{CK}, Ichida \cite{I}) (see Subsection \ref{sec:Variants of splitting theorems}).

\subsection{Inscribed radii}
We denote by $M^{n}_{\kappa}$ the simply connected $n$-dimensional space form with constant curvature $\kappa$.
We say that $\kappa$ and $\lambda$ satisfy the \textit{ball-condition}
if there exists a closed geodesic ball $\ball$ in $M^{n}_{\kappa}$ whose boundary $\bball$ has constant mean curvature $(n-1)\lambda$.
Notice that $\kappa$ and $\lambda$ satisfy the ball-condition if and only if either
(1) $\kappa>0$; 
(2) $\kappa=0$ and $\lambda>0$;
or (3) $\kappa<0$ and $\lambda>\sqrt{\vert \kappa \vert}$.
We denote by $\const$ the radius of $\ball$.
We see that $\kappa$ and $\lambda$ satisfy the ball-condition if and only if
the equation $s_{\kappa,\lambda}(t)=0$ has a positive solution;
moreover,
$\const=\inf \{\,t>0 \mid s_{\kappa,\lambda}(t)=0\, \}$.

The \textit{inscribed radius $\IR M$ of $M$} is defined to be the supremum of the distance function from the boundary $\rho_{\bm}$ over $M$.
Let us consider the Riemannian metric $g_{f}:=e^{\frac{-4f}{n-1}}g$.
We denote by $\rho^{g_{f}}_{\bm}$ and by $\IR_{g_{f}} M$
the distance function from the boundary and the inscribed radius on $M$ induced from $g_{f}$,
respectively.

Let $\inte M$ be the interior of $M$.
For $x \in \inte M$,
let $U_{x}M$ be the unit tangent sphere at $x$ which
can be identified with the $(n-1)$-dimensional standard unit sphere $(\mathbb{S}^{n-1},ds^{2}_{n-1})$.
For $v \in U_{x} M$,
let $\gamma_{v}:[0,T)\to M$ be the geodesic with $\gamma_{v}'(0)=v$. 
We define $\tau_{x}:U_{x}M\to (0,\infty]$ by
\begin{equation}\label{eq:pointed cut value}
\tau_{x}(v):=\sup \left\{\, t>0 \mid \rho_{x}(\gamma_{v}(t))=t,\, \gamma_{v}([0,t))\subset \inte M  \,\right\},
\end{equation}
where $\rho_{x}:M\to \mathbb{R}$ is the distance function from $x$ defined as $\rho_{x}(y):=d_{M}(x,y)$.
Let $s_{f,v}:[0,\tau_{x}(v)]\to [0,\infty]$ denote a function defined by
\begin{equation}\label{eq:pointed speed changing function}
s_{f,v}(t):=\int^{t}_{0}\,  e^{\frac{-2f(\gamma_{v}(a))}{n-1}}\,da.
\end{equation}
Let $s_{\kappa}(s)$ be a unique solution of the Jacobi equation $\varphi''(s)+\kappa \varphi(s)=0$ with $\varphi(0)=0,\,\varphi'(0)=1$.
For $t\in [0,\tau_{x}(v)]$
we put
\begin{equation}\label{eq:pointed model Jacobi field}
F_{\kappa,v}(t):=\exp \left(\frac{f(\gamma_{v}(t))+f(x)}{n-1}   \right) \,s_{\kappa}(s_{f,v}(t)).
\end{equation}
For $l>0$,
we denote by $[0,l]\times_{F_{\kappa}} \mathbb{S}^{n-1}$ the twisted product Riemannian manifold $\bigl([0,l]\times \mathbb{S}^{n-1},dt^{2}+F^{2}_{\kappa,v}(t)ds^{2}_{n-1}\bigl)$.

We further prove the following inscribed radius rigidity theorem:
\begin{thm}\label{thm:inscribed radius rigidity}
Let us assume that
$\kappa$ and $\lambda$ satisfy the ball-condition.
For $N \in (-\infty,1]$,
assume that
$(M,\bm,f)$ has lower $(\kappa,\lambda,N)$-weighted curvature bounds.
Then we have
\begin{equation}\label{eq:inscribed radius rigidity}
\IR_{g_{f}} M \leq \const.
\end{equation}
If $\rho^{g_{f}}_{\bm}(x_{0})=\const$ for some $x_{0}\in M$,
then $M$ is isometric to $[0,l]\times_{F_{\kappa}} \mathbb{S}^{n-1}$ for some $l>0$;
moreover,
if $N\in (-\infty,1)$,
then $f$ is constant;
in particular,
$M$ is isometric to a closed ball in a space form.
\end{thm}
Kasue \cite{K2} has proved Theorem \ref{thm:inscribed radius rigidity} in the case of $f=0$.

We will also obtain an inscribed radius rigidity theorem for $\IR M$ in the case where $f$ is bounded from above (see Theorem \ref{thm:finite inscribed radius rigidity}).

\subsection{Volume growths}
We set $\bar{C}_{\kappa,\lambda}:=\const$
if $\kappa$ and $\lambda$ satisfy the ball-condition;
otherwise,
$\bar{C}_{\kappa,\lambda}:=\infty$.
We define functions $\bar{s}_{\kappa,\lambda},\,s_{n,\kappa,\lambda}:[0,\infty)\to \mathbb{R}$ by
\begin{equation}\label{eq:model volume growth}
\bar{s}_{\kappa,\lambda}(s):=\begin{cases}
                                                    s_{\kappa,\lambda}(s) & \text{if $s<\bar{C}_{\kappa,\lambda}$}, \\
                                                                                0           & \text{if $s \geq \bar{C}_{\kappa,\lambda}$},
                                                   \end{cases}
\quad
s_{n,\kappa,\lambda}(r):=\int^{r}_{0}\,\bar{s}^{n-1}_{\kappa,\lambda}(a)\,da.                                          
\end{equation}

For a smooth function $\phi:M\to \mathbb{R}$
we define
\begin{equation*}
m_{\phi}:=e^{-\phi} \vol_{g}.
\end{equation*}
For $x \in M$,
we say that
$z \in \bm$ is a \textit{foot point on $\bm$ of $x$} if $d_{M}(x,z)=\rho_{\bm}(x)$.
Every point in $M$ has at least one foot point on $\bm$.
Let us define a function $\rho_{\bm,f}:M\to \mathbb{R}$ by
\begin{equation}\label{eq:weighted distance function from the boundary}
\rho_{\bm,f}(x):=\inf_{z} \, \int^{\rho_{\bm}(x)}_{0}\, e^{\frac{-2f(\gamma_{z}(a))}{n-1}}\,da,
\end{equation}
where the infimum is taken over all foot points $z\in \bm$ of $x$.
For $r>0$,
\begin{equation}\label{eq:twisted inscribed radius}
B^{f}_{r}(\bm):=\{\,x\in M \mid \rho_{\bm,f}(x) \leq r \,\},\quad \IR_{f} M:=\sup_{x\in M} \rho_{\bm,f}(x).
\end{equation}

We prove absolute volume comparisons of Heintze-Karcher type (\cite{HK}),
and relative volume comparisons (see Subsections \ref{sec:Absolute volume comparisons} and \ref{sec:Relative volume comparisons}).

One of the relative volume comparison theorems is the following:
\begin{thm}\label{thm:relative volume comparison}
For $N \in (-\infty,1]$,
assume that
$(M,\bm,f)$ has lower $(\kappa,\lambda,N)$-weighted curvature bounds.
Let $\bm$ be compact.
Then for all $r,R>0$ with $r\leq R$
we have
\begin{equation}\label{eq:relative volume comparison}
\frac{m_{\frac{n+1}{n-1}f}(B^{f}_{R}(\bm))}{m_{\frac{n+1}{n-1}f}(B^{f}_{r}(\bm))} \leq \frac{s_{n,\kappa,\lambda}(R)}{s_{n,\kappa,\lambda}(r)}.
\end{equation}
\end{thm}

We provide a rigidity theorem concerning the equality case of Theorem \ref{thm:relative volume comparison} (see Theorem \ref{thm:volume growth rigidity}).
We also present a volume growth rigidity theorem in the case where $f$ is bounded from above (see Theorem \ref{thm:finite volume growth rigidity}).

\subsection{Eigenvalues}
For $p\in [1,\infty)$,
and a smooth function $\phi:M\to \mathbb{R}$,
let $W^{1,p}_{0}(M,m_{\phi})$ stand for the \textit{$(1,p)$-Sobolev space with compact support}
defined as the completion of $C^{\infty}_{0}(M)$ with respect to the standard $(1,p)$-Sobolev norm.
The \textit{$(\phi,p)$-Laplacian} is defined as
\begin{equation*}
\Delta_{\phi,p}:=-e^{\phi}\,\Div \,\left(e^{-\phi} \Vert \nabla \cdot \Vert^{p-2}\, \nabla \cdot \right)
\end{equation*}
as a distribution on $W^{1,p}_{0}(M,m_{\phi})$,
where $\Vert \cdot \Vert$ is the standard norm,
and $\Div$ is the divergence with respect to $g$.
A real number $\nu$ is said to be a \textit{$(\phi,p)$-Dirichlet eigenvalue} on $M$
if there exists $\psi \in W^{1,p}_{0}(M,m_{\phi}) \setminus \{0\}$
such that $\Delta_{\phi,p} \psi=\nu \vert \psi \vert^{p-2}\,\psi$ holds in the distribution sense.
For $\psi \in W^{1,p}_{0}(M,m_{\phi})\setminus \{0\}$,
the \textit{Rayleigh quotient} is defined as
\begin{equation*}
R_{\phi,p}(\psi):=\frac{\int_{M}\, \Vert \nabla \psi \Vert^{p}\,d\,m_{\phi}}{\int_{M}\,  \vert \psi \vert^{p}\,d\,m_{\phi}}.
\end{equation*}
We study
\begin{equation*}
\nu_{\phi,p}(M):=\inf_{\psi} R_{\phi,p}(\psi),
\end{equation*}
where the infimum is taken over all $\psi \in W^{1,p}_{0}(M,m_{\phi}) \setminus \{0\}$.
If $M$ is compact,
and if $p\in (1,\infty)$,
then $\nu_{\phi,p}(M)$ is equal to the infimum of the set of all $(\phi,p)$-Dirichlet eigenvalues on $M$.

Let $p\in (1,\infty)$.
For $D\in (0,\bar{C}_{\kappa,\lambda}]\setminus \{\infty\}$,
let $\nu_{p,\kappa,\lambda,D}$ be the positive minimum real number $\nu$ such that
there exists a non-zero function $\varphi:[0,D]\to \mathbb{R}$ satisfying
\begin{align}\label{eq:model space eigenvalue problem}
\left(\vert \varphi'(s)\vert^{p-2} \varphi'(s)\right)'&+(n-1)\frac{s'_{\kappa,\lambda}(s)}{s_{\kappa,\lambda}(s)}\,\left(\vert \varphi'(s) \vert^{p-2} \varphi'(s)\right)\\ \notag
                                                                         &+\nu\, \vert \varphi(s)\vert^{p-2}\varphi(s)=0, \,\, \varphi(0)=0,\,\, \varphi'(D)=0.                           
\end{align}

Let us recall the notion of the model spaces introduced by Kasue \cite{K3}.
We say that
$\kappa$ and $\lambda$ satisfy the \textit{model-condition}
if the equation $s'_{\kappa,\lambda}(t)=0$ has a positive solution.
Note that
$\kappa$ and $\lambda$ satisfy the model-condition if and only if either
(1) $\kappa>0$ and $\lambda<0$;
(2) $\kappa=0$ and $\lambda=0$;
or (3) $\kappa<0$ and $\lambda \in (0,\sqrt{\vert \kappa \vert})$.
Let $\kappa$ and $\lambda$ satisfy the ball-condition or the model-condition,
and let $M$ be compact.
For an interval $I$,
and for a connected component $\bm_{1}$ of $\bm$,
we denote by $I \times_{\kappa,\lambda} \bm_{1}$ the warped product Riemannian manifold $\left(I \times \bm_{1},dt^{2}+s^{2}_{\kappa,\lambda}(t)\,h\right)$.
When $\kappa$ and $\lambda$ satisfy the model-condition,
we define $D_{\kappa,\lambda}(M)$ as follows:
If $\kappa=0$ and $\lambda=0$,
then $D_{\kappa,\lambda}(M):=\IR M$;
otherwise,
$D_{\kappa,\lambda}(M):=\inf \{t>0\mid s'_{\kappa,\lambda}(t)=0\}$.
We say that $M$ is a \textit{$\left(\kappa,\lambda \right)$-equational model space}
if $M$ is isometric to either
(1) the closed geodesic ball $\ball$ for $\kappa$ and $\lambda$ satisfying the ball-condition;
(2) the warped product $[0,2D_{\kappa,\lambda}(M)]\times_{\kappa,\lambda} \bm_{1}$ for $\kappa$ and $\lambda$ satisfying the model-condition,
and for some connected component $\bm_{1}$ of $\bm$;
or (3) the quotient space $\left([0,2D_{\kappa,\lambda}(M)]\times_{\kappa,\lambda} \bm \right)/G_{\sigma}$ for $\kappa$ and $\lambda$ satisfying the model-condition,
and for some involutive isometry $\sigma$ of $\bm$ without fixed points,
where $G_{\sigma}$ denotes the isometry group on $[0,2D_{\kappa,\lambda}(M)]\times_{\kappa,\lambda} \bm$
whose elements consist of identity and the involute isometry $\hat{\sigma}$ defined by $\hat{\sigma}(t,z):=\left(2D_{\kappa,\lambda}(M)-t,\sigma(z)\right)$.

We say that
$f$ is \textit{$\bm$-radial}
if there exists a smooth function $\phi_{f}:[0,\infty)\to \mathbb{R}$ such that
$f=\phi_{f}\circ \rho_{\bm}$ on $M$.

We establish the following theorem for the smallest eigenvalue $\nu_{\frac{n+1}{n-1}f,p}$:
\begin{thm}\label{thm:eigenvalue rigidity}
Let $p\in (1,\infty)$.
For $N \in (-\infty,1]$,
let us assume that
$(M,\bm,f)$ has lower $(\kappa,\lambda,N)$-weighted curvature bounds.
Let $M$ be compact,
and let $f$ be $\bm$-radial.
Suppose additionally that
there exists $\delta \in \mathbb{R}$ such that $f \leq (n-1)\delta$ on $M$.
For $D\in (0,\bar{C}_{\kappa,\lambda}]\setminus \{\infty\}$,
suppose $\IR_{f} M \leq D$,
where $\IR_{f} M$ is defined as $(\ref{eq:twisted inscribed radius})$.
Then
\begin{equation}\label{eq:eigenvalue rigidity}
\nu_{\frac{n+1}{n-1}f,p}(M)\geq \nu_{p,\kappa\,e^{-4\delta},\lambda\,e^{-2\delta},D\,e^{2\delta}}.
\end{equation}
If the equality in $(\ref{eq:eigenvalue rigidity})$ holds,
then $M$ is a $\left(\kappa e^{-4\delta} ,\lambda e^{-2\delta} \right)$-equational model space,
and $f=(n-1)\delta$ on $M$.
\end{thm}
In the case where $f=0$ and $\delta=0$,
Kasue \cite{K3} has proved Theorem \ref{thm:eigenvalue rigidity} for $p=2$,
and the author \cite{Sa2} has done for any $p\in (1,\infty)$.

We also formulate a rigidity theorem for the smallest eigenvalue $\nu_{f,p}$ in the case where $f$ is not necessarily $\bm$-radial (see Theorem \ref{thm:finite eigenvalue rigidity}).
Furthermore,
we obtain a spectrum rigidity theorem for complete (not necessarily compact) manifolds with boundary (see Theorem \ref{thm:spectrum rigidity}).

\subsection{Organization}
In Section \ref{sec:Preliminaries},
we prepare some notations
and recall the basic facts for Riemannian manifolds with boundary.
We also recall the works done by Wylie and Yeroshkin \cite{WY} (see Subsection \ref{sec:Laplacian comparisons from a single point}).

In Sections \ref{sec:Laplacian comparisons} and \ref{sec:Global Laplacian comparisons},
to prove our main theorems,
we study Laplacian comparisons for the distance function from the boundary.
In Section \ref{sec:Laplacian comparisons},
we show a pointwise Laplacian comparison result (see Subsection \ref{sec:Basic Laplacian comparisons}),
and a rigidity result in the equality case (see Subsection \ref{sec:Equality cases}).
In Section \ref{sec:Global Laplacian comparisons},
we prove global Laplacian comparison inequalities in the distribution sense in the case where $f$ is bounded from above (see Subsection \ref{sec:Bounded cases}),
and where $f$ is $\bm$-radial (see Subsection \ref{sec:Radial cases}).

In Section \ref{sec:Splitting theorems},
we prove splitting theorems.
In Section \ref{sec:Inscribed radii},
we examine inscribed radius rigidity theorems.
In Section \ref{sec:Volume growths},
we show volume comparison theorems.
In Section \ref{sec:Eigenvalues},
we study eigenvalue rigidity theorems.

\subsection*{{\rm Acknowledgements}}
The author would like to express his gratitude to Koichi Nagano for his constant advice and suggestions.
The author would also like to thank William Wylie for informing him of the paper \cite{WY}.
The author is grateful to Shin-ichi Ohta for his comments.

\section{Preliminaries}\label{sec:Preliminaries}
We refer to \cite{S} for the basics of Riemannian manifolds with boundary (see also Section 2 in \cite{Sa1}, \cite{Sa2} and \cite{Sa3}).

\subsection{Riemannian manifolds with boundary} 
Let $M$ be a connected Riemannian manifold with boundary.
For $r>0$ and $A\subset M$,
we denote by $B_{r}(A)$ the closed $r$-neighborhood of $A$.
For $A_{1}, \,A_{2}\subset M$,
let $d_{M}(A_{1},A_{2})$ denote the distance between them.
For an interval $I$,
we say that
a curve $\gamma:I\to M$ is a \textit{minimal geodesic}
if for all $t_{1},t_{2}\in I$
we have $d_{M}(\gamma(t_{1}),\gamma(t_{2}))=\vert t_{1}-t_{2}\vert$.
If the metric space $(M,d_{M})$ is complete,
then the Hopf-Rinow theorem for length spaces (see e.g., Theorem 2.5.23 in \cite{BBI}) tells us that
it is a proper geodesic space.

For $i=1,2$,
let $(M_{i},g_{i})$ be connected Riemannian manifolds with boundary.
For each $i$,
the boundary $\bm_{i}$ carries the induced metric $h_{i}$.
We say that a homeomorphism $\Phi:M_{1}\to M_{2}$ is a \textit{Riemannian isometry with boundary} if $\Phi$ satisfies the following conditions:
\begin{enumerate}
 \item $\Phi|_{\inte M_{1}}:\inte M_{1} \to \inte M_{2}$ is smooth, and $(\Phi|_{\inte M_{1}})^{\ast} (g_{2})=g_{1}$;\label{enum:inner isom}
 \item $\Phi|_{\bm_{1}}:\bm_{1} \to \bm_{2}$ is smooth, and $(\Phi|_{\bm_{1}})^{\ast} (h_{2})=h_{1}$.\label{enum:bdry isom}
\end{enumerate}
There exists a Riemannian isometry with boundary from $M_{1}$ to $M_{2}$ if and only if
$(M_{1},d_{M_{1}})$ and $(M_{2},d_{M_{2}})$ are isometric to each other.

\subsection{Jacobi fields}\label{sec:Jacobi fields orthogonal to the boundary}
Let $(M,g)$ be a connected Riemannian manifold with boundary.
For $x\in \inte M$, 
let $T_{x}M$ and $U_{x}M$ be the tangent space and the unit tangent sphere at $x$,
respectively.
For $z\in \bm$,
and the tangent space $T_{z}\bm$ at $z$ on $\bm$,
let $T_{z}^{\perp} \bm$ be its orthogonal complement in the tangent space at $z$ on $M$.

For vector fields $v_{1},\,v_{2}$ on $\bm$,
the second fundamental form $S(v_{1},v_{2})$ is defined as
the normal component of $\nabla^{g}_{v_{1}}v_{2}$ with respect to $\bm$,
where $\nabla^{g}$ is the Levi-Civita connection induced from $g$.
For $u\in T_{z}^{\perp}\bm$,
the shape operator $A_{u}:T_{z}\bm \to T_{z}\bm$ is defined as
\begin{equation*}
g(A_{u}v_{1},v_{2}):=g(S(v_{1},v_{2}),u).
\end{equation*}
For the unit inner normal vector $u_{z}$ on $\bm$ at $z$,
the mean curvature $H_{z}$ of $\bm$ at $z$ is defined as the trace of $A_{u_{z}}$.
We say that a Jacobi field $Y$ along the geodesic $\gamma_{z}$ is a $\bm$-\textit{Jacobi field} if $Y$ satisfies
\begin{equation*}
Y(0)\in T_{z}\bm, \quad Y'(0)+A_{u_{z}}Y(0)\in T_{z}^{\perp}\bm.
\end{equation*}
We say that $\gamma_{z}(t_{0})$ is a \textit{conjugate point} of $\bm$ along $\gamma_{z}$
if there exists a non-zero $\bm$-Jacobi field $Y$ along $\gamma_{z}$ such that $Y(t_{0})=0$.

\subsection{Cut locus for the boundary}\label{sec:Cut locus for the boundary}
We recall the basic properties of the cut locus for the boundary.
We refer to \cite{Sa1} for the proofs.

Let $(M,g)$ be a connected complete Riemannian manifold with boundary.
For $x \in \inte M$, 
let $z\in \bm$ be a foot point on $\bm$ of $x$ (i.e., $d_{M}(x,z)=\rho_{\bm}(x)$).
In this case,
there exists a unique minimal geodesic $\gamma:[0,l]\to M$ from $z$ to $x$ such that $\gamma=\gamma_{z}|_{[0,l]}$,
where $l=\rho_{\bm}(x)$.
In particular,
$\gamma'(0)=u_{z}$ and $\gamma|_{(0,l]}$ lies in $\inte M$.

Let $\tau:\bm\to (0,\infty]$ be the function defined as (\ref{eq:boundary cut value}).
The supremum of $\tau$ over $\bm$ is equal to the inscribed radius $\IR M$.
The function $\tau$ is continuous on $\bm$.
The continuity of $\tau$ tells us that
if $\bm$ is compact,
then $\IR M<\infty$ if and only if $M$ is compact.

The \textit{cut locus for the boundary} is defined as
\begin{equation*}
\cut \bm:=\{ \,\gamma_{z}(\tau(z)) \mid z\in \bm,\,\tau(z)<\infty \,\}.
\end{equation*}
From the continuity of $\tau$
we see that
$\cut \bm$ is a null set of $M$.
For $x \in \inte M\setminus \cut \bm$,
its foot point on $\bm$ is uniquely determined.

In \cite{Sa1},
we have already known the following:
\begin{lem}\label{lem:splitting lemma}
If there exists a connected component $\bm_{0}$ of $\bm$ such that $\tau = \infty$ on $\bm_{0}$,
then $\bm$ is connected and $\cut \bm=\emptyset$.
\end{lem}
The function $\rho_{\bm}$ is smooth on $\inte M\setminus \cut \bm$.
For each $x\in \inte M\setminus \cut \bm$,
we have $\nabla \rho_{\bm}(x)=\gamma'(l)$,
where $\gamma:[0,l]\to M$ is the minimal geodesic from the foot point on $\bm$ of $x$ to $x$.

For $\Omega \subset M$,
we denote by $\bar{\Omega}$ its closure,
and by $\partial \Omega$ its boundary.
For a domain $\Omega$ in $M$ such that
$\partial \Omega$ is a smooth hypersurface in $M$,
we denote by $\vol_{\partial \Omega}$ the canonical Riemannian volume measure on $\partial \Omega$.

We recall the following fact to avoid the cut locus for the boundary (see Lemma 2.6 in \cite{Sa2}):
\begin{lem}\label{lem:avoiding the cut locus}
Let $\Omega \subset M$ be a domain such that
$\partial \Omega$ is a smooth hypersurface in $M$.
Then there exists a sequence $\{\Omega_{i}\}$ of closed subsets of $\bar{\Omega}$ such that
for every $i$,
the set $\partial \Omega_{i}$ is a smooth hypersurface in $M$ except for a null set in $(\partial \Omega,\vol_{\partial \Omega})$
satisfying the following properties:
\begin{enumerate}
\item for all $i_{1},i_{2}$ with $i_{1}<i_{2}$,
         we have $\Omega_{i_{1}}\subset \Omega_{i_{2}}$;
\item $\bar{\Omega} \setminus \cut \bm=\bigcup_{i}\,\Omega_{i}$;
\item for every $i$,
         and for almost every $x \in \partial \Omega_{i}\cap \partial \Omega$ in $(\partial \Omega,\vol_{\partial \Omega})$,
         there exists a unique unit outer normal vector for $\Omega_{i}$ at $x$
         that coincides with the unit outer normal vector on $\partial \Omega$ for $\Omega$ at $x$;
\item for every $i$,
         on $\partial \Omega_{i}\setminus \partial \Omega$,
         there exists a unique unit outer normal vector field $u_{i}$ for $\Omega_{i}$ such that $g(u_{i},\nabla \rho_{\bm})\geq 0$.
\end{enumerate}
Moreover,
if $\bar{\Omega}=M$,
then for every $i$,
the set $\partial \Omega_{i}$ is a smooth hypersurface in $M$,
and satisfies $\partial \Omega_{i}\cap \bm=\bm$.
\end{lem}

\subsection{Busemann functions and asymptotes}
Let $M$ be a connected complete Riemannian manifold with boundary.
A minimal geodesic $\gamma:[0,\infty)\to M$ is said to be a \textit{ray}.
For a ray $\gamma:[0,\infty)\to M$, 
the \textit{Busemann function} $b_{\gamma}:M\to \mathbb{R}$ is defined as
\begin{equation*}
b_{\gamma}(x):=\lim_{t\to \infty}(t-d_{M}(x,\gamma(t))).
\end{equation*}

We have the following lemma (see Lemma 6.1 in \cite{Sa1}):
\begin{lem}\label{lem:busemann function}
Suppose that for some $z\in \bm$
we have $\tau(z)=\infty$.
Take $x\in \inte M$.
If $b_{\gamma_{z}}(x)=\rho_{\bm}(x)$,
then $x\notin \cut \bm$.
Moreover,
for the unique foot point $z_{x}$ on $\bm$ of $x$,
we have $\tau(z_{x})=\infty$.
\end{lem}

Let $\gamma:[0,\infty) \to M$ be a ray.
For $x \in M$,
we say that
a ray $\gamma_{x}:[0,\infty)\to M$ is an \textit{asymptote for $\gamma$ from $x$}
if there exists a sequence $\{t_{i}\}$ with $t_{i}\to \infty$ such that
the following holds:
For each $i$,
there exists a minimal geodesic $\gamma_{i}:[0,l_{i}] \to M$ from $x$ to $\gamma(t_{i})$ such that
for every $t \geq 0$
we have $\gamma_{i}(t)\to \gamma_{x}(t)$ as $i \to \infty$.
Since $M$ is proper,
for each $x\in M$
there exists at least one asymptote for $\gamma$ from $x$.

For asymptotes,
we see the following (see Lemma 6.2 in \cite{Sa1}):
\begin{lem}\label{lem:asymptote}
Suppose that for some $z\in \bm$
we have $\tau(z)=\infty$.
For $l>0$,
put $x:=\gamma_{z}(l)$.
Then there exists $\epsilon>0$ such that
for all $y\in B_{\epsilon}(x)$,
all asymptotes for the ray $\gamma_{z}$ from $y$ lie in $\inte M$.
\end{lem}

\subsection{Weighted manifolds with boundary}
Let $(M,g)$ be a connected complete Riemannian manifold with boundary,
and let $f:M\to \mathbb{R}$ be a smooth function.
The \textit{weighted Laplacian $\Delta_{f}$} is defined by
\begin{equation*}\label{eq:weighted Laplacian}
\Delta_{f}:=\Delta+g(\nabla f,\nabla \cdot),
\end{equation*}
where $\Delta$ is the Laplacian defined as the minus of the trace of the Hessian.
Note that
$\Delta_{f}$ coincides with the $(f,2)$-Laplacian $\Delta_{f,2}$.

The following formula of Bochner type is well-known (see e.g., \cite{V}).
\begin{prop}\label{prop:Bochner formula}
For every smooth function $\psi$ on $M$,
we have
\begin{equation*}
-\frac{1}{2}\,\Delta_{f}\,\Vert \nabla \psi \Vert^{2}=\ric^{\infty}_{f}(\nabla \psi)+\Vert \Hess \psi \Vert^{2}-g\left(\nabla \Delta_{f}\,\psi,\nabla \psi \right),
\end{equation*}
where $\Vert \Hess \psi \Vert$ is the Hilbert-Schmidt norm of $\Hess \psi$.
\end{prop}
For $z\in \bm$,
the value $\Delta_{f}\rho_{\bm}(\gamma_{z}(t))$ tends to $H_{f,z}$ as $t\to 0$.
For $t \in (0,\tau(z))$,
and for the volume element $\theta(t,z)$ of the $t$-level surface of $\rho_{\bm}$ at $\gamma_{z}(t)$,
we put
\begin{equation}\label{eq:volume element}
\theta_{f}(t,z):=e^{-f(\gamma_{z}(t))}\, \theta(t,z).
\end{equation}
For all $t\in (0,\tau(z))$
it holds that
\begin{equation}\label{eq:Laplacian representation}
\Delta_{f}\, \rho_{\bm}(\gamma_{z}(t)) =-\frac{\theta_{f}'(t,z)}{\theta_{f}(t,z)}.
\end{equation}
We further define a function $\bar{\theta}_{f}:[0,\infty) \times \bm \to \mathbb{R}$ by
\begin{equation*}\label{eq:volume element}
  \bar{\theta}_{f}(t,z) := \begin{cases}
                                     \theta_{f}(t,z) & \text{if $t< \tau(z)$}, \\
                                             0           & \text{if $t \geq \tau(z)$}.
                                    \end{cases}
\end{equation*}

The following has been shown in \cite{Sa2}:
\begin{lem}\label{lem:Basic integration formula}
If $\bm$ is compact,
then for all $r>0$
\begin{equation*}
m_{f}( B_{r}(\bm))=\int_{\bm} \int^{r}_{0}\bar{\theta}_{f}(t,z)\,dt\,d\vol_{h},
\end{equation*}
where $\vol_{h}$ is the Riemannian volume measure on $\bm$ induced from $h$.
\end{lem}

Let $\psi:M\to \mathbb{R}$ be a continuous function,
and let $U$ be a domain contained in $\inte M$.
For $x \in U$,
and for a function $\hat{\psi}$ defined on an open neighborhood of $x$,
we say that
$\hat{\psi}$ is a \textit{support function of $\psi$ at $x$}
if we have $\hat{\psi}(x)=\psi(x)$ and $\hat{\psi} \leq \psi$.
We say that
$\psi$ is \textit{$f$-subharmonic on $U$}
if for every $x \in U$,
and for every $\epsilon>0$,
there exists a smooth support function $\psi_{x,\epsilon}$ of $\psi$ at $x$
such that $\Delta_{f}\, \psi_{x,\epsilon}(x)\leq \epsilon$.

We recall the following maximal principle (see e.g., \cite{C}):
\begin{lem}[\cite{C}]\label{lem:maximal principle}
Let $U$ be a domain contained in $\inte M$.
If an $f$-subharmonic function on $U$ takes the maximal value at a point in $U$,
then it must be constant on $U$.
\end{lem}

\subsection{Laplacian comparisons from a single point}\label{sec:Laplacian comparisons from a single point}
We recall the works done by Wylie and Yeroshkin \cite{WY}.
Let $M$ be an $n$-dimensional,
connected complete Riemannian manifold with boundary,
and let $f:M\to \mathbb{R}$ be a smooth function.
For the diameter $C_{\kappa}$ of the space form $M^{n}_{\kappa}$,
we define a function $H_{\kappa}:(0,C_{\kappa})\to \mathbb{R}$ by
\begin{equation}\label{eq:pointed model mean curvature}
H_{\kappa}(s):=-(n-1)\frac{s'_{\kappa}(s)}{s_{\kappa}(s)}.
\end{equation}

Wylie and Yeroshkin \cite{WY} have proved a Laplacian comparison inequality for the distance function from a single point (see Theorem 4.4 in \cite{WY}).
In our setting,
the inequality holds in the following form:
\begin{lem}[\cite{WY}]\label{lem:pointed Laplacian comparison}
Let $x \in \inte M$ and $v \in U_{x}M$.
For $N \in (-\infty,1]$,
assume $\ric^{N}_{f,M}\geq (n-1)\kappa\, e^{\frac{-4f}{n-1}}$.
Then for all $t \in (0,\tau_{x}(v))$
\begin{equation}\label{eq:pointed Laplacian comparison}
\Delta_{f}\, \rho_{x}(\gamma_{v}(t))\geq H_{\kappa}(s_{f,v}(t))\,e^{\frac{-2f(\gamma_{v}(t))}{n-1}},
\end{equation}
where $\tau_{x}$ and $s_{f,v}$ are defined as $(\ref{eq:pointed cut value})$ and as $(\ref{eq:pointed speed changing function})$,
respectively.
\end{lem}

As a corollary of the Laplacian comparison inequality,
Wylie and Yeroshkin \cite{WY} have shown another Laplacian comparison inequality in the case where $f$ is bounded (see Corollary 4.11 in \cite{WY}).
In our setting,
by using the same method of the proof,
we see the following:
\begin{lem}[\cite{WY}]\label{lem:finite pointed Laplacian comparison}
Let $x \in \inte M$ and $v \in U_{x}M$.
For $N \in (-\infty,1]$,
assume $\ric^{N}_{f,M}\geq (n-1)\kappa\, e^{\frac{-4f}{n-1}}$.
Suppose additionally that
there is $\delta \in \mathbb{R}$ such that $f\leq (n-1)\delta$ on $M$.
Then for all $t \in (0,\tau_{x}(v))$
we have
\begin{equation}\label{eq:finite pointed Laplacian comparison}
\Delta_{f}\, \rho_{x}(\gamma_{v}(t))\geq H_{\kappa}\left(e^{-2\delta}t\right)\,e^{\frac{-2f(\gamma_{v}(t))}{n-1}}.
\end{equation}
\end{lem}
\begin{proof}
From $f \leq (n-1)\delta$,
we deduce $s_{f,v}(t)\geq e^{-2\delta}t$ for every $t \in (0,\tau_{x}(v))$.
We see $H'_{\kappa}>0$ on $(0,C_{\kappa})$,
and hence (\ref{eq:pointed Laplacian comparison}) implies
\begin{equation}\label{eq:proof of finite pointed Laplacian comparison}
\Delta_{f}\, \rho_{x}(\gamma_{v}(t))\geq H_{\kappa}(s_{f,v}(t))\,e^{\frac{-2f(\gamma_{v}(t))}{n-1}} \geq H_{\kappa}\left(e^{-2\delta}t\right)\,e^{\frac{-2f(\gamma_{v}(t))}{n-1}}.
\end{equation}
This proves (\ref{eq:finite pointed Laplacian comparison}).
\end{proof}

Wylie and Yeroshkin \cite{WY} have proved a rigidity result in the equality case of the Laplacian comparison inequality (see Lemma 4.13 in \cite{WY}).
From the argument discussed in their proof,
one can derive:
\begin{lem}[\cite{WY}]\label{lem:equality in pointed Laplacian comparison}
Under the same setting as in Lemma \ref{lem:pointed Laplacian comparison},
assume that
for some $t_{0}\in (0,\tau_{x}(v))$
the equality in $(\ref{eq:pointed Laplacian comparison})$ holds.
Choose an orthonormal basis $\{e_{v,i}\}_{i=1}^{n}$ of $T_{x}M$ with $e_{v,n}=v$.
Let $\{Y_{v,i}\}^{n-1}_{i=1}$ be the Jacobi fields along $\gamma_{v}$ with $Y_{v,i}(0)=0_{x},\,Y_{v,i}'(0)=e_{v,i}$.
Then for all $i$
we have $Y_{v,i}=F_{\kappa,v}\,E_{v,i}$ on $[0,t_{0}]$,
where $F_{\kappa,v}$ is defined as $(\ref{eq:pointed model Jacobi field})$,
and $\{E_{v,i}\}^{n-1}_{i=1}$ are the parallel vector fields with $E_{v,i}(0)=e_{v,i}$;
moreover,
if $N\in (-\infty,1)$,
then $f \circ \gamma_{v}$ is constant on $[0,t_{0}]$.
\end{lem}
\begin{rem}\label{rem:equality in finite pointed Laplacian comparison}
Under the same setting as in Lemma \ref{lem:finite pointed Laplacian comparison},
assume that
for some $t_{0}\in (0,\tau_{x}(v))$
the equality in $(\ref{eq:finite pointed Laplacian comparison})$ holds.
Then the equalities in (\ref{eq:proof of finite pointed Laplacian comparison}) hold.
In particular,
the equality in (\ref{eq:pointed Laplacian comparison}) holds (see Lemma \ref{lem:equality in pointed Laplacian comparison}),
and $s_{f,v}(t_{0})=e^{-2\delta}t_{0}$,
and hence $f \circ \gamma_{v}=(n-1)\delta$ on $[0,t_{0}]$.
\end{rem}

\section{Laplacian comparisons}\label{sec:Laplacian comparisons}
Hereafter,
let $(M,g)$ be an $n$-dimensional, 
connected complete Riemannian manifold with boundary,
and let $f:M\to \mathbb{R}$ be smooth.

\subsection{Basic Laplacian comparisons}\label{sec:Basic Laplacian comparisons}
For the distance function from a single point,
Wylie and Yeroshkin \cite{WY} have shown an inequality of Riccati type (see Lemma 4.1 in \cite{WY}).
By the same method of the proof,
for the distance function from the boundary,
we have the following:
\begin{lem}\label{lem:Riccati}
Let $z \in \bm$ and $N \in (-\infty,1]$.
Then for all $t \in (0,\tau(z))$
\begin{align}\label{eq:Riccati}
&\quad  \left(\bigl(e^{\frac{2f}{n-1}}\,\Delta_{f}\rho_{\bm}\bigl)(\gamma_{z}(t))\right)'\\ \notag
&\geq      \ric^{N}_{f}(\gamma'_{z}(t))\, e^{\frac{2f(\gamma_{z}(t))}{n-1}}+\frac{ \left( \bigl(e^{\frac{2f}{n-1}}\,\Delta_{f}\rho_{\bm}\bigl)(\gamma_{z}(t))   \right)^{2}   }{n-1}\,e^{\frac{-2f(\gamma_{z}(t))}{n-1}}.
\end{align}
\end{lem}
\begin{proof}
Put $f_{z}:=f \circ \gamma_{z}$ and $h_{f,z}:=\left(\Delta_{f}\rho_{\bm}\right) \circ \gamma_{z}$.
Applying Proposition \ref{prop:Bochner formula} to the distance function $\rho_{\bm}$,
we have
\begin{align*}
0 &= \ric^{\infty}_{f}(\gamma'_{z}(t))+\Vert \Hess \rho_{\bm}\Vert^{2}\left(\gamma_{z}(t)\right)-g\left(\nabla \Delta_{f}\rho_{\bm},\nabla \rho_{\bm}  \right)(\gamma_{z}(t))\\
   &= \left(\ric^{N}_{f}(\gamma'_{z}(t))+\frac{f'_{z}(t)^{2}}{N-n}\right)+\Vert \Hess \rho_{\bm}\Vert^{2}\left(\gamma_{z}(t)\right)-h'_{f,z}(t). \notag
\end{align*}
By the Cauchy-Schwarz inequality,
\begin{equation}\label{eq:Cauchy-Schwarz inequality}
\Vert \Hess \rho_{\bm}\Vert^{2}\left(\gamma_{z}(t)\right) \geq \frac{\left(\Delta \rho_{\bm}(\gamma_{z}(t))\right)^{2}}{n-1}=\frac{\left( h_{f,z}(t)-f'_{z}(t)\right)^{2}}{n-1}.
\end{equation}
The inequality (\ref{eq:Cauchy-Schwarz inequality}) yields
\begin{align}\label{eq:combination}
0 &\geq \ric^{N}_{f}(\gamma'_{z}(t))+\frac{f'_{z}(t)^{2}}{N-n}+\frac{\left(h_{f,z}(t)-f'_{z}(t)\right)^{2}}{n-1}-h'_{f,z}(t)\\ \notag
   &   =   \ric^{N}_{f}(\gamma'_{z}(t))+\frac{\left(1-N\right)f'_{z}(t)^{2}}{\left(n-1\right)\left(n-N\right)}+\frac{h^{2}_{f,z}(t)}{n-1}\\ \notag
   &\qquad \qquad \qquad \,\,  -\left(\frac{2h_{f,z}(t)\,f'_{z}(t)}{n-1}+h'_{f,z}(t)\right).
\end{align}
The last term in the right hand side of (\ref{eq:combination}) satisfies
\begin{equation*}
\frac{2h_{f,z}(t)  f'_{z}(t)}{n-1}+h'_{f,z}(t)=e^{\frac{-2f(\gamma_{z}(t))}{n-1}}\,\left(   e^{\frac{2f(\gamma_{z}(t))}{n-1}}  \,h_{f,z}(t)\right)'.
\end{equation*}
We put $F_{z}(t):=e^{\frac{2f(\gamma_{z}(t))}{n-1}}  \,h_{f,z}(t)$.
From $N \in (-\infty,1]$
it follows that
\begin{align}\label{eq:at most 1}
0 &\geq \ric^{N}_{f}(\gamma'_{z}(t))+\frac{\left(1-N\right)f'_{z}(t)^{2}}{\left(n-1\right)\left(n-N\right)}+\frac{h^{2}_{f,z}(t)}{n-1}-e^{\frac{-2f(\gamma_{z}(t))}{n-1}}\,F'_{z}(t)\\ \notag
   &\geq \ric^{N}_{f}(\gamma'_{z}(t))+\frac{h^{2}_{f,z}(t)}{n-1}-e^{\frac{-2f(\gamma_{z}(t))}{n-1}}\,F'_{z}(t).
\end{align}
This implies
\begin{align*}
F'_{z}(t) &\geq e^{\frac{2f(\gamma_{z}(t))}{n-1}}\,\left(\ric^{N}_{f}(\gamma'_{z}(t))+\frac{h^{2}_{f,z}(t)}{n-1}\right)\\
             &   =   \ric^{N}_{f}(\gamma'_{z}(t))\,e^{\frac{2f(\gamma_{z}(t))}{n-1}}+\frac{F^{2}_{z}(t)}{n-1}\,e^{\frac{-2f(\gamma_{z}(t))}{n-1}}.
\end{align*}
We arrive at the desired inequality (\ref{eq:Riccati}).
\end{proof}
\begin{rem}\label{rem:equality in Riccati}
We assume that
the equality in (\ref{eq:Riccati}) holds for some $t_{0}\in (0,\tau(z))$.
Then the equality in the Cauchy-Schwarz inequality in (\ref{eq:Cauchy-Schwarz inequality}) holds;
in particular,
there exists a constant $c$ such that
$\Hess \rho_{\bm}=c\,g$ on the orthogonal complement of $\nabla \rho_{\bm}$ at $\gamma_{z}(t_{0})$.
Moreover,
the equalities in (\ref{eq:at most 1}) hold;
in particular,
$(1-N)f'_{z}(t_{0})^{2}=0$.
\end{rem}

Recall that
$\tau_{f}$ and $s_{f,z}$ are defined as (\ref{eq:boundary cut value}) and as (\ref{eq:speed changing function}),
respectively.
We denote by $t_{f,z}:[0,\tau_{f}(z)]\to [0,\tau(z)]$ the inverse function of $s_{f,z}$.

We define a function $H_{\kappa,\lambda}:[0,\bar{C}_{\kappa,\lambda})\to \mathbb{R}$ by
\begin{equation}\label{eq:model mean curvature}
H_{\kappa,\lambda}(s):=-(n-1)\frac{s'_{\kappa,\lambda}(s)}{s_{\kappa,\lambda}(s)}.
\end{equation}
For all $s\in [0,\bar{C}_{\kappa,\lambda})$
we see
\begin{equation}\label{eq:model Riccati}
H'_{\kappa,\lambda}(s)=(n-1)\kappa+\frac{H^{2}_{\kappa,\lambda}(s)}{n-1}.
\end{equation}

We prove the following pointwise Laplacian comparison inequality:
\begin{lem}\label{lem:Basic comparison}
Let $z\in \bm$.
For $N \in (-\infty,1]$,
let us assume that
$\ric^{N}_{f}(\gamma'_{z}(t))\geq (n-1)\kappa\,e^{\frac{-4f(\gamma_{z}(t))}{n-1}}$ for all $t \in (0,\tau(z))$,
and $H_{f,z}\geq (n-1)\lambda \,e^{\frac{-2 f(z)}{n-1}}$.
Then for all $s \in (0,\min\{\tau_{f}(z) ,\bar{C}_{\kappa,\lambda} \}  )$ we have
\begin{equation}\label{eq:Basic comparison}
\Delta_{f}\rho_{\bm}\left(\gamma_{z}\left(    t_{f,z}(s)  \right) \right) \geq H_{\kappa,\lambda}(s) \, e^{\frac{-2f\left(\gamma_{z}\left(    t_{f,z}(s)  \right)\right) }{n-1}}.
\end{equation}
In particular,
for all $t \in (0,\tau(z))$ with $s_{f,z}(t) \in (0,\min\{\tau_{f}(z) ,\bar{C}_{\kappa,\lambda} \})$
\begin{equation}\label{eq:Laplacian comparison}
\Delta_{f}\rho_{\bm}(\gamma_{z}(t)) \geq H_{\kappa,\lambda}(s_{f,z}(t))\,e^{\frac{-2f(\gamma_{z}(t))}{n-1}}\,.
\end{equation}
\end{lem}
\begin{proof}
We define a function $F_{z}:(0,\tau(z)) \to \mathbb{R}$ by 
\begin{equation*}
F_{z}:=\bigl(e^{\frac{2f}{n-1}}\,\Delta_{f}\rho_{\bm}\bigl) \circ \gamma_{z},
\end{equation*}
and a function $\hat{F}_{z}:(0,\tau_{f}(z)) \to \mathbb{R}$ by $\hat{F}_{z}:=F_{z}\circ t_{f,z}$.
By Lemma \ref{lem:Riccati} and the curvature assumption,
for all $s \in (0,\tau_{f}(z))$,
\begin{align}\label{eq:use of Riccati}
\hat{F}'_{z}(s)&    =   F'_{z}(t_{f,z}(s))\, e^{\frac{2f\left(\gamma_{z}\left(    t_{f,z}(s)  \right)\right) }{n-1}}\\ \notag
                      & \geq \ric^{N}_{f}(\gamma'_{z}(t_{f,z}(s)))\,e^{\frac{4f\left(\gamma_{z}\left(    t_{f,z}(s)  \right)\right) }{n-1}}+\frac{F^{2}_{z}(t_{f,z}(s))}{n-1}\\\notag
                      & \geq (n-1)\kappa+\frac{\hat{F}^{2}_{z}(s)}{n-1}.
\end{align}
The identity (\ref{eq:model Riccati}) implies that
for all $s \in (0,\min\{\tau_{f}(z) ,\bar{C}_{\kappa,\lambda} \}  )$
\begin{equation}\label{eq:sharp Riccati}
\hat{F}'_{z}(s)-H'_{\kappa,\lambda}(s)\geq \frac{\hat{F}^{2}_{z}(s)-H^{2}_{\kappa,\lambda}(s)}{n-1}.
\end{equation}

Let us define a function $G_{\kappa,\lambda,z}:(0,\min\{\tau_{f}(z) ,\bar{C}_{\kappa,\lambda} \}  )\to \mathbb{R}$ by
\begin{equation*}
G_{\kappa,\lambda,z}:=s^{2}_{\kappa,\lambda}\bigl( \hat{F}_{z}-H_{\kappa,\lambda} \bigl).
\end{equation*}
From (\ref{eq:sharp Riccati})
we deduce
\begin{align}\label{eq:monotonicity}
G'_{\kappa,\lambda,z} &   =  2s_{\kappa,\lambda}s'_{\kappa,\lambda}\bigl( \hat{F}_{z}-H_{\kappa,\lambda} \bigl)+s^{2}_{\kappa,\lambda}\bigl( \hat{F}'_{z}-H'_{\kappa,\lambda} \bigl)\\ \notag
                                    &\geq 2s_{\kappa,\lambda}s'_{\kappa,\lambda}\bigl(\hat{F}_{z}-H_{\kappa,\lambda} \bigl)+s^{2}_{\kappa,\lambda}\frac{\hat{F}^{2}_{z}-H^{2}_{\kappa,\lambda}}{n-1}\\ \notag
                                    &  =   \frac{s^{2}_{\kappa,\lambda}}{n-1}\bigl(\hat{F}_{z}-H_{\kappa,\lambda} \bigl)^{2}\geq 0.   
\end{align}
Since $G_{\kappa,\lambda,z}(s)$ converges to a non-negative value $e^{\frac{2 f(z)}{n-1}}H_{f,z}-(n-1)\lambda$ as $s\to 0$,
the function $G_{\kappa,\lambda,z}$ is non-negative.
We conclude that
$\hat{F}_{z} \geq H_{\kappa,\lambda}$ holds on $(0,\min\{\tau_{f}(z) ,\bar{C}_{\kappa,\lambda} \}  )$,
and hence (\ref{eq:Basic comparison}).
\end{proof}
\begin{rem}\label{rem:equality in the Basic comparison}
We assume that
the equality in $(\ref{eq:Basic comparison})$ holds for some $s_{0}\in (0,\min\{\tau_{f}(z) ,\bar{C}_{\kappa,\lambda} \}  )$.
Then we have $G_{\kappa,\lambda,z}(s_{0})=0$.
From $G'_{\kappa,\lambda,z} \geq 0$
it follows that $G_{\kappa,\lambda,z}=0$ on $[0,s_{0}]$;
in particular,
the equality in (\ref{eq:Basic comparison}) holds on $[0,s_{0}]$.
Since the equalities in (\ref{eq:use of Riccati}), (\ref{eq:sharp Riccati}), (\ref{eq:monotonicity}) hold,
the equality in (\ref{eq:Riccati}) holds on $[0,t_{f,z}(s_{0})]$ (see Remark \ref{rem:equality in Riccati}).
\end{rem}

From Lemma \ref{lem:Basic comparison}
we derive the following estimate for $\tau_{f}$:
\begin{lem}\label{lem:Cut point comparisons}
Let $z \in \bm$.
Let $\kappa$ and $\lambda$ satisfy the ball-condition.
For $N\in (-\infty,1]$,
we assume that
$\ric^{N}_{f}(\gamma'_{z}(t))\geq (n-1)\kappa\,e^{\frac{-4f(\gamma_{z}(t))}{n-1}}$ for all $t \in (0,\tau(z))$,
and $H_{f,z}\geq (n-1)\lambda e^{\frac{-2 f(z)}{n-1}}$.
Then we have
\begin{equation}\label{eq:Cut point comparisons}
\tau_{f}(z) \leq \const.
\end{equation}
Moreover,
if there is $\delta \in \mathbb{R}$ such that $f\circ \gamma_{z} \leq (n-1)\delta$ on $(0,\tau(z))$,
\begin{equation}\label{eq:finite Cut point comparisons}
\tau(z)\leq C_{\kappa\,e^{-4\delta},\lambda\,e^{-2\delta}}.
\end{equation}
\end{lem}
\begin{proof}
The proof is by contradiction.
Suppose $\tau_{f}(z)>\const$.
Then we see $\tau(z)>t_{f,z}(\const)$.
By (\ref{eq:Laplacian comparison}),
for every $t \in (0,t_{f,z}(\const))$
\begin{equation*}
\Delta_{f}\rho_{\bm}(\gamma_{z}(t)) \geq H_{\kappa,\lambda}(s_{f,z}(t))\,e^{\frac{-2f(\gamma_{z}(t))}{n-1}};
\end{equation*}
in particular,
$\Delta_{f}\rho_{\bm}(\gamma_{z}(t)) \to \infty$ as $t\to t_{f,z}(\const)$.
This contradicts the smoothness of $\rho_{\bm} \circ \gamma_{z}$ on $(0,\tau(z))$.
It follows that (\ref{eq:Cut point comparisons}).
If $f\circ \gamma_{z} \leq (n-1)\delta$,
then we have $e^{-2\delta}\,\tau(z) \leq \tau_{f}(z)$.
By $e^{2\delta}C_{\kappa,\lambda}=C_{\kappa\,e^{-4\delta},\lambda\,e^{-2\delta}}$,
we arrive at the desired inequality (\ref{eq:finite Cut point comparisons}).
\end{proof}
\begin{rem}
Lemma \ref{lem:Cut point comparisons} enables us to restate the conclusion of Lemma \ref{lem:Basic comparison} as follows:
For all $s \in (0,\tau_{f}(z))$
we have (\ref{eq:Basic comparison}).
In particular,
for all $t\in (0,\tau(z))$
we have (\ref{eq:Laplacian comparison}).
\end{rem}

\subsection{Equality cases}\label{sec:Equality cases}
Recall the following (see e.g., Theorem 2 in \cite{Pe}):
\begin{lem}\label{lem:radial curvature equation}
Let $\rho$ be a smooth function defined on a domain in $M$ such that $\Vert \nabla \rho \Vert=1$.
Let $X$ be a parallel vector field along an integral curve of $\nabla \rho$
that is orthogonal to $\nabla \rho$.
Then we have
\begin{equation*}\label{eq:radial curvature equation}
g(R(X,\nabla \rho)\nabla \rho,X)=g(\nabla_{\nabla \rho}  A_{\nabla \rho}X,X)-g(A_{\nabla \rho}A_{\nabla \rho}X,X),
\end{equation*}
where $R$ is the curvature tensor induced from $g$,
and $A_{\nabla \rho}$ is the shape operator of the level set of $\rho$ toward $\nabla \rho$.
In particular,
if there exists a function $\varphi$ defined on the domain of the integral curve such that $A_{\nabla \rho}X=-\varphi\,X$,
then $g(R(X,\nabla \rho)\nabla \rho,X)=-(\varphi'+\varphi^{2})\Vert X \Vert^{2}$.
\end{lem}

For the equality case of (\ref{eq:Basic comparison}) in Lemma \ref{lem:Basic comparison},
we have:
\begin{lem}\label{lem:Equality in Basic comparison}
Under the same setting as in Lemma \ref{lem:Basic comparison},
assume that
for some $s_{0}\in (0, \tau_{f}(z)  )$
the equality in $(\ref{eq:Basic comparison})$ holds.
Choose an orthonormal basis $\{e_{z,i}\}_{i=1}^{n-1}$ of $T_{z}\bm$,
and let $\{Y_{z,i}\}^{n-1}_{i=1}$ be the $\bm$-Jacobi fields along $\gamma_{z}$
with $Y_{z,i}(0)=e_{z,i},\,Y_{z,i}'(0)=-A_{u_{z}}e_{z,i}$.
Then for all $i$ we have $Y_{z,i}=F_{\kappa,\lambda,z}\,E_{z,i}$ on $[0,t_{f,z}(s_{0})]$,
where $F_{\kappa,\lambda,z}$ is defined as $(\ref{eq:model Jacobi field})$,
and $\{E_{z,i}\}^{n-1}_{i=1}$ are the parallel vector fields with $E_{z,i}(0)=e_{z,i}$;
moreover,
if $N \in (-\infty,1)$,
then $f \circ \gamma_{z}$ is constant on $[0,t_{f,z}(s_{0})]$.
\end{lem}
\begin{proof}
Put $t_{0}:=t_{f,z}(s_{0})$.
Since the equality in $(\ref{eq:Basic comparison})$ holds at $s_{0}$,
the equality in (\ref{eq:Riccati}) also holds on $[0,t_{0}]$ (see Remark \ref{rem:equality in Riccati}).
There exists a function $\varphi$ on the set $\gamma_{z}((0,t_{0}))$ such that
at each point on $\gamma_{z}((0,t_{0}))$,
we have $\Hess \rho_{\bm}=\varphi\,g$ on the orthogonal complement of $\nabla \rho_{\bm}$ (see Remark \ref{rem:equality in the Basic comparison}).
Define $\varphi_{z}:=\varphi \circ \gamma_{z}$.
For each $i$,
it holds that
\begin{equation*}
g(A_{\nabla \rho_{\bm}}E_{z,i},E_{z,i})=-\Hess \rho_{\bm}(E_{z,i},E_{z,i})=-\varphi_{z};
\end{equation*}
in particular,
$A_{\nabla \rho_{\bm}}E_{z,i}=-\varphi_{z}\,E_{z,i}$.
From Lemma \ref{lem:radial curvature equation}
we deduce
\begin{equation}\label{eq:curvature}
R(E_{z,i},\nabla \rho_{\bm})\nabla \rho_{\bm}=-(\varphi'_{z}+\varphi^{2}_{z})E_{z,i}=-\mathcal{F}''_{z}\,\mathcal{F}^{-1}_{z}\,E_{z,i},
\end{equation}
where $\mathcal{F}_{z}:[0,t_{0}] \to \mathbb{R}$ is a function defined by 
\begin{equation*}
\mathcal{F}_{z}(t):= \exp \left( \int^{t}_{0}\, \varphi_{z}(a)\,da \right).
\end{equation*}

Set $f_{z}:=f \circ \gamma_{z}$ and $h_{f,z}:=\left(\Delta_{f}\rho_{\bm}\right) \circ \gamma_{z}$.
By the equality assumption,
$e^{\frac{2f_{z}}{n-1}}\,h_{f,z}=H_{\kappa,\lambda} \circ s_{f,z}$ on $[0,t_{0}]$ (see Remarks \ref{rem:equality in Riccati} and \ref{rem:equality in the Basic comparison}).
Furthermore,
$\Hess \rho_{\bm}=\varphi\,g$ leads to $\Delta\, \rho_{\bm} \circ \gamma_{z}=-(n-1)\varphi_{z}$.
Therefore,
\begin{equation*}
\varphi_{z}(t)=\frac{1}{n-1}  \left(f_{z}(t)-\int^{t}_{0}\,e^{\frac{-2f(\gamma_{z}(a))}{n-1}}\,\left(H_{\kappa,\lambda}\circ s_{f,z}\right)(a)\,da \right)'
\end{equation*}
for every $t\in [0,t_{0}]$.
It follows that
$\mathcal{F}_{z}=F_{\kappa,\lambda,z}$ on $[0,t_{0}]$.
In view of (\ref{eq:curvature}),
we obtain $Y_{z,i}=F_{\kappa,\lambda,z}\,E_{z,i}$ on $[0,t_{0}]$.

We have $(1-N)(f'_{z})^{2}=0$ on $[0,t_{0}]$ (see Remarks \ref{rem:equality in Riccati} and \ref{rem:equality in the Basic comparison}).
If $N\in (-\infty,1)$,
then $f'_{z}=0$ on $[0,t_{0}]$;
in particular,
$f_{z}$ is constant.
\end{proof}

For the equality case of (\ref{eq:Laplacian comparison}),
Lemma \ref{lem:Equality in Basic comparison} implies:
\begin{lem}\label{lem:Equality in Laplacian comparison}
Under the same setting as in Lemma \ref{lem:Basic comparison},
assume that
for some $t_{0}\in (0,\tau(z))$
the equality in $(\ref{eq:Laplacian comparison})$ holds.
Choose an orthonormal basis $\{e_{z,i}\}_{i=1}^{n-1}$ of $T_{z}\bm$,
and let $\{Y_{z,i}\}^{n-1}_{i=1}$ be the $\bm$-Jacobi fields along $\gamma_{z}$ with $Y_{z,i}(0)=e_{z,i},\,Y_{z,i}'(0)=-A_{u_{z}}e_{z,i}$.
Then for all $i$
we have $Y_{z,i}=F_{\kappa,\lambda,z}\,E_{z,i}$ on $[0,t_{0}]$,
where $\{E_{z,i}\}^{n-1}_{i=1}$ are the parallel vector fields with $E_{z,i}(0)=e_{z,i}$;
moreover,
if $N \in (-\infty,1)$,
then $f \circ \gamma_{z}$ is constant.
\end{lem}
\section{Global Laplacian comparisons}\label{sec:Global Laplacian comparisons}
We start with introducing some conditions.
Let us recall that
$\kappa$ and $\lambda$ satisfy the ball-condition if and only if either
(1) $\kappa>0$; 
(2) $\kappa=0$ and $\lambda>0$;
or (3) $\kappa<0$ and $\lambda>\sqrt{\vert \kappa \vert}$.
We say that $\kappa$ and $\lambda$ satisfy the \textit{convex-ball-condition}
if they satisfy the ball-condition and $\lambda \geq 0$.

Furthermore,
we say that
$\kappa$ and $\lambda$ satisfy the \textit{monotone-condition}
if $H_{\kappa,\lambda}\geq 0$ and $H'_{\kappa,\lambda}\geq 0$ on $[0,\bar{C}_{\kappa,\lambda})$,
where $H_{\kappa,\lambda}$ is defined as (\ref{eq:model mean curvature}).
We see that
$\kappa$ and $\lambda$ satisfy the monotone-condition
if and only if either
(1) $\kappa$ and $\lambda$ satisfy the convex-ball-condition;
or (2) $\kappa \leq 0$ and $\lambda=\sqrt{\vert \kappa \vert}$.
For $\kappa$ and $\lambda$ satisfying the monotone-condition,
if $\kappa = 0$ and $\lambda=0$,
then $H_{\kappa,\lambda}=0$ on $[0,\infty)$;
otherwise,
$H_{\kappa,\lambda} > 0$ on $(0,\bar{C}_{\kappa,\lambda})$.

We also say that
$\kappa$ and $\lambda$ satisfy the \textit{weakly-monotone-condition}
if $H'_{\kappa,\lambda} \geq 0$ on $[0,\bar{C}_{\kappa,\lambda})$.
Notice that
$\kappa$ and $\lambda$ satisfy the weakly-monotone-condition
if and only if either
(1) $\kappa \geq 0$;
or (2) $\kappa<0$ and $\vert \lambda \vert \geq \sqrt{\vert \kappa \vert}$.
In particular,
if $\kappa$ and $\lambda$ satisfy the ball-condition,
then they also satisfy the weakly-monotone-condition.
For $\kappa$ and $\lambda$ satisfying the weakly-monotone-condition,
if $\kappa \leq 0$ and $\vert \lambda \vert=\sqrt{\vert \kappa \vert}$,
then $H_{\kappa,\lambda}=(n-1)\lambda$ on $[0,\infty)$;
otherwise,
$H'_{\kappa,\lambda} > 0$ on $[0,\bar{C}_{\kappa,\lambda})$.

\subsection{Bounded cases}\label{sec:Bounded cases}
If $f$ is bounded from above,
then we have:
\begin{lem}\label{lem:finite Laplacian comparison}
Let $z\in \bm$.
Let $\kappa$ and $\lambda$ satisfy the weakly-monotone-condition.
For $N \in (-\infty,1]$,
let us assume that
$\ric^{N}_{f}(\gamma'_{z}(t))\geq (n-1)\kappa\, e^{\frac{-4 f(\gamma_{z}(t))}{n-1}}$ for all $t \in (0,\tau(z))$,
and $H_{f,z}\geq (n-1)\lambda e^{\frac{-2 f(z)}{n-1}}$.
Suppose additionally that
there exists $\delta \in \mathbb{R}$ such that $f\circ \gamma_{z}\leq (n-1)\delta$ on $(0,\tau(z))$.
Then for all $t \in (0,\tau(z))$
we have
\begin{equation}\label{eq:finite Laplacian comparison}
\Delta_{f}\rho_{\bm}(\gamma_{z}(t))  \geq H_{\kappa,\lambda}(e^{-2\delta}t)\,  e^{\frac{-2 f(\gamma_{z}(t))}{n-1}}.
\end{equation}
Moreover,
if $\kappa$ and $\lambda$ satisfy the monotone-condition,
then
\begin{equation}\label{eq:strictly finite Laplacian comparison}
\Delta_{f}\rho_{\bm}(\gamma_{z}(t))  \geq H_{\kappa,\lambda}(e^{-2\delta}t)\,e^{-2\delta}.
\end{equation}
\end{lem}
\begin{proof}
By $f\circ \gamma_{z}\leq (n-1)\delta$,
it holds that $s_{f,z}(t)\geq e^{-2\delta}t$ and $e^{\frac{-2 f(\gamma_{z}(t))}{n-1}} \geq e^{-2\delta}$ for all $t\in (0,\tau(z))$.
The inequality (\ref{eq:Laplacian comparison}) and $H'_{\kappa,\lambda}\geq 0$ tell us that
\begin{equation}\label{eq:proof of finite Laplacian comparison}
\Delta_{f}\rho_{\bm}(\gamma_{z}(t)) \geq H_{\kappa,\lambda}(s_{f,z}(t))\,e^{\frac{-2 f(\gamma_{z}(t))}{n-1}} \geq H_{\kappa,\lambda}(e^{-2\delta}t)\,e^{\frac{-2 f(\gamma_{z}(t))}{n-1}}
\end{equation}
for all $t \in (0,\tau(z))$,
and hence (\ref{eq:finite Laplacian comparison}).
Moreover,
if $\kappa$ and $\lambda$ satisfy the monotone-condition,
then (\ref{eq:finite Laplacian comparison}) and $H_{\kappa,\lambda} \geq 0$ lead to
\begin{equation}\label{eq:proof of strictly finite Laplacian comparison}
\Delta_{f}\rho_{\bm}(\gamma_{z}(t)) \geq H_{\kappa,\lambda}(e^{-2\delta}t)\,e^{\frac{-2 f(\gamma_{z}(t))}{n-1}} \geq H_{\kappa,\lambda}(e^{-2\delta}t)\,e^{-2\delta}.
\end{equation}
This proves (\ref{eq:strictly finite Laplacian comparison}).
\end{proof}

\begin{rem}\label{rem:equality in finite Laplacian comparison}
Assume that
for some $t_{0}\in (0,\tau(z))$
the equality in $(\ref{eq:finite Laplacian comparison})$ holds.
Then the equalities in (\ref{eq:proof of finite Laplacian comparison}) hold,
and the equality in (\ref{eq:Laplacian comparison}) also holds (see Lemma \ref{lem:Equality in Laplacian comparison}).
Moreover,
if either (1) $\kappa>0$; or (2) $\kappa \leq 0$ and $\vert \lambda \vert >\sqrt{\vert \kappa \vert}$,
then $H'_{\kappa,\lambda} > 0$ on $[0,\bar{C}_{\kappa,\lambda})$,
and hence $s_{f,z}(t_{0})=e^{-2\delta}t_{0}$;
in particular,
$f\circ \gamma_{z}=(n-1)\delta$ on $[0,t_{0}]$.
\end{rem}

\begin{rem}\label{rem:equality in strictly finite Laplacian comparison}
Assume that
for some $t_{0}\in (0,\tau(z))$
the equality in $(\ref{eq:strictly finite Laplacian comparison})$ holds.
Then the equalities in (\ref{eq:proof of strictly finite Laplacian comparison}) hold,
and the equality in (\ref{eq:finite Laplacian comparison}) holds (see Remark \ref{rem:equality in finite Laplacian comparison}).
Moreover,
if either (1) $\kappa$ and $\lambda$ satisfy the convex-ball-condition; or (2) $\kappa<0$ and $\lambda=\sqrt{\vert \kappa \vert}$,
then $H_{\kappa,\lambda} > 0$ on $(0,\bar{C}_{\kappa,\lambda})$,
and hence $e^{\frac{-2 f(\gamma_{z}(t_{0}))}{n-1}}=e^{-2\delta}$;
in particular,
$(f\circ \gamma_{z})(t_{0})=(n-1)\delta$.
\end{rem}

Lemma \ref{lem:finite Laplacian comparison} implies the following:
\begin{lem}\label{lem:finite p-Laplacian comparison}
Let $z\in \bm$ and $p\in (1,\infty)$.
Let $\kappa$ and $\lambda$ satisfy the monotone-condition.
For $N\in (-\infty,1]$,
we assume that
$\ric^{N}_{f}(\gamma'_{z}(t))\geq (n-1)\kappa\,e^{\frac{-4 f(\gamma_{z}(t))}{n-1}}$ for all $t \in (0,\tau(z))$,
and $H_{f,z}\geq (n-1)\lambda e^{\frac{-2 f(z)}{n-1}}$.
Suppose additionally that
there exists $\delta \in \mathbb{R}$ such that $f\circ \gamma_{z} \leq (n-1)\delta$ on $(0,\tau(z))$.
We define $\rho_{\bm,\delta}:=e^{-2\delta}\rho_{\bm}$.
Let $\varphi:[0,\infty)\to \mathbb{R}$ be a monotone increasing smooth function.
Then for all $t \in (0,\tau(z))$
\begin{align}\label{eq:finite p-Laplacian comparison}
&\quad \, \Delta_{f,p}\, \bigl(\varphi \circ \rho_{\bm,\delta} \bigl) (\gamma_{z}(t)) \\ \notag
&\geq -e^{-2p\delta} \left\{   \left(   \bigl(     (  \varphi'  )^{p-1}     \bigl)'-H_{\kappa,\lambda}\,(\varphi')^{p-1} \right)\circ \rho_{\bm,\delta} \right\}(\gamma_{z}(t)).
\end{align}
\end{lem}
\begin{proof} 
Set $\Phi:=\varphi \circ \rho_{\bm,\delta}$,
and define $\varphi_{\delta}(t):=\varphi(e^{-2\delta}t)$.
We see $\Phi=\varphi_{\delta}\circ \rho_{\bm}$.
By (\ref{eq:strictly finite Laplacian comparison}),
for every $t\in (0,\tau(z))$
\begin{align}\label{eq:pre finite p-Laplacian comparison}
\Delta_{f,p}\, \Phi (\gamma_{z}(t))&   =   - \bigl(     (  \varphi'_{\delta}  )^{p-1}     \bigl)'(t)  +\Delta_{f,2}\rho_{\bm}(\gamma_{z}(t))\, \left(\varphi'_{\delta}\right)^{p-1}(t)\\ \notag
                                                     &\geq - \bigl(     (  \varphi'_{\delta}  )^{p-1}     \bigl)'(t)  +H_{\kappa,\lambda}(e^{-2\delta}t)\,e^{-2\delta}\, \left(\varphi'_{\delta}\right)^{p-1}(t).
\end{align}
Since $\left(\varphi'_{\delta}\right)^{p-1}(t) =e^{-2(p-1)\delta}\,\left(\varphi'\right)^{p-1}(e^{-2\delta}t)$ and 
\begin{equation*}
\bigl(     (  \varphi'_{\delta}  )^{p-1}     \bigl)'(t) =e^{-2p\delta}\bigl(     (  \varphi'  )^{p-1}     \bigl)'(e^{-2\delta}t),
\end{equation*}
the right hand side of (\ref{eq:pre finite p-Laplacian comparison}) is equal to that of (\ref{eq:finite p-Laplacian comparison}).
\end{proof}

\begin{rem}\label{rem:equality in finite p-Laplacian comparison}
The equality case of Lemma \ref{lem:finite p-Laplacian comparison}
results in that of (\ref{eq:strictly finite Laplacian comparison})
(see Remark \ref{rem:equality in strictly finite Laplacian comparison}).
\end{rem}

We now prove the following global Laplacian comparison inequality:
\begin{prop}\label{prop:global finite p-Laplacian comparison}
Let $p\in (1,\infty)$.
Let $\kappa$ and $\lambda$ satisfy the monotone-condition.
For $N\in (-\infty,1]$,
assume that
the triple $(M,\bm,f)$ has lower $(\kappa,\lambda,N)$-weighted curvature bounds.
We suppose additionally that
there exists $\delta \in \mathbb{R}$ such that $f \leq (n-1)\delta$ on $M$.
We define $\rho_{\bm,\delta}:=e^{-2\delta}\,\rho_{\bm}$.
Let $\varphi:[0,\infty)\to \mathbb{R}$ be a monotone increasing smooth function.
Then we have
\begin{equation*}
\Delta_{f,p}\, \left(\varphi \circ \rho_{\bm,\delta} \right) \geq -e^{-2p\delta} \left( \bigl( (\varphi' )^{p-1} \bigl)' -H_{\kappa,\lambda}\,  \left(\varphi' \right)^{p-1}\right)\circ \rho_{\bm,\delta}
\end{equation*}
in the following distribution sense on $M$:
For every non-negative function $\psi \in C^{\infty}_{0}(M)$ we have
\begin{align}\label{eq:global finite p-Laplacian comparison}
&\quad \int_{M}\,   \Vert \nabla \bigl(   \varphi \circ \rho_{\bm,\delta} \bigl) \Vert^{p-2}\, g\left(\nabla \psi, \nabla \left(   \varphi \circ \rho_{\bm,\delta} \right)  \right)\, d\,m_{f}\\ \notag
&\geq -e^{-2p\delta}\, \int_{M}\,\psi\,   \left\{ \left( \bigl( (\varphi' )^{p-1} \bigl)' -H_{\kappa,\lambda}\,\left(\varphi' \right)^{p-1}\right)\circ \rho_{\bm,\delta} \right\}  \,d\,m_{f}.
\end{align}
\end{prop}
\begin{proof}
By Lemma \ref{lem:avoiding the cut locus},
there exists a sequence $\{\Omega_{i}\}$ of closed subsets of $M$
such that
for every $i$,
the set $\partial \Omega_{i}$ is a smooth hypersurface in $M$,
and satisfying the following:
(1) for all $i_{1},i_{2}$ with $i_{1}<i_{2}$,
we have $\Omega_{i_{1}}\subset \Omega_{i_{2}}$;
(2) $M\setminus \cut \bm=\bigcup_{i}\,\Omega_{i}$;
(3) $\partial \Omega_{i}\cap \bm=\bm$ for all $i$;
(4) for each $i$,
     on $\partial \Omega_{i}\setminus \bm$,
     there exists a unique unit outer normal vector field $u_{i}$ for $\Omega_{i}$ with $g(u_{i},\nabla \rho_{\bm})\geq 0$.

For the canonical volume measure $\vol_{i}$ on $\partial \Omega_{i}\setminus \bm$,
put $m_{f,i}:=e^{  -f|_{\partial \Omega_{i}\setminus \bm}}\,\vol_{i}$.
Set $\Phi:=\varphi \circ \rho_{\bm,\delta}$.
By integration by parts,
we see
\begin{align*}
&\quad \,  \int_{\Omega_{i}}\,  \Vert \nabla \Phi \Vert^{p-2} \,g\left(\nabla \psi, \nabla \Phi  \right)  \, d\,m_{f}\\
&=\int_{\Omega_{i}}\, \psi\,\Delta_{f,p} \Phi \,d\,m_{f}+\int_{\partial \Omega_{i}\setminus \bm}\, \Vert \nabla \Phi \Vert^{p-2}\, \psi \,g\left(u_{i},\nabla \Phi \right) \, d\,m_{f,i}.
\end{align*}
From (\ref{eq:finite p-Laplacian comparison}) and $g(u_{i},\nabla \rho_{\bm,\delta})\geq 0$,
it follows that
the right hand side of the above equality is at least
\begin{align*}
-e^{-2p\delta}\int_{\Omega_{i}}\,\psi\,  \left\{  \left( \bigl( (\varphi' )^{p-1} \bigl)' -H_{\kappa,\lambda}\, \left(\varphi' \right)^{p-1}\right)\circ \rho_{\bm,\delta}  \right\}  \,d\,m_{f}.
\end{align*}
Letting $i\to \infty$,
we obtain (\ref{eq:global finite p-Laplacian comparison}).
\end{proof}

\begin{rem}\label{rem:equality case in global finite p-Laplacian comparison}
In Proposition \ref{prop:global finite p-Laplacian comparison},
assume that
the equality in (\ref{eq:global finite p-Laplacian comparison}) holds.
In this case,
the equality in (\ref{eq:finite p-Laplacian comparison}) also holds on $\supp \psi \setminus (\bm \cup \cut \bm)$,
where $\supp \psi$ denotes the support of $\psi$.
The equality case of Proposition \ref{prop:global finite p-Laplacian comparison}
results in that of (\ref{eq:finite p-Laplacian comparison}) (see Remark \ref{rem:equality in finite p-Laplacian comparison}).
\end{rem}

\begin{rem}\label{rem:c1-equality case in global finite p-Laplacian comparison}
The argument in the proof of Proposition \ref{prop:global finite p-Laplacian comparison} also tells us the following (see also Remark \ref{rem:equality case in global finite p-Laplacian comparison}):
Under the same setting as in Proposition \ref{prop:global finite p-Laplacian comparison},
if $M$ is compact,
then the inequality (\ref{eq:global finite p-Laplacian comparison}) holds  for every non-negative function $\psi \in C^{1}(M)$ with $\psi|_{\bm}=0$.
Moreover,
if the equality in (\ref{eq:global finite p-Laplacian comparison}) holds for some $\psi$,
then the equality in (\ref{eq:finite p-Laplacian comparison}) holds on $\supp \psi \setminus (\bm \cup \cut \bm)$ (see Remark \ref{rem:equality in finite p-Laplacian comparison}).
\end{rem}

\subsection{Radial cases}\label{sec:Radial cases}
Suppose that
$f$ is $\bm$-radial.
Then there exists a smooth function $\phi_{f}:[0,\infty)\to \mathbb{R}$ such that $f=\phi_{f} \circ \rho_{\bm}$ on $M$.
Define a function $s_{f}:[0,\infty]\to [0,\infty]$ by
\begin{equation}\label{eq:radial parameter function}
s_{f}(t):=\int^{t}_{0} \, e^{\frac{-2\phi_{f}(a)}{n-1}} \,da.
\end{equation}
For every $z\in \bm$,
we see $s_{f,z}=s_{f}$ on $[0,\tau(z)]$,
where $s_{f,z}$ is defined as (\ref{eq:speed changing function}).
Furthermore,
$\rho_{\bm,f}=s_{f}\circ \rho_{\bm}$,
where $\rho_{\bm,f}$ is defined as (\ref{eq:weighted distance function from the boundary}).

If $f$ is $\bm$-radial,
then we have the following comparison inequality:
\begin{lem}\label{lem:radial p-Laplacian comparison}
Let $z\in \bm$ and $p\in (1,\infty)$.
For $N\in (-\infty,1]$,
suppose that
$\ric^{N}_{f}(\gamma'_{z}(t))\geq (n-1)\kappa\,e^{\frac{-4 f(\gamma_{z}(t))}{n-1}}$ for all $t \in (0,\tau(z))$,
and $H_{f,z}\geq (n-1)\lambda e^{\frac{-2 f(z)}{n-1}}$.
Suppose that
$f$ is $\bm$-radial.
Let $\varphi:[0,\infty)\to \mathbb{R}$ be a monotone increasing smooth function.
Then for all $t \in (0,\tau(z))$
\begin{align}\label{eq:radial p-Laplacian comparison}
&\quad \, \Delta_{\frac{n+1-2p}{n-1}f,p}\, \bigl( \varphi \circ \rho_{\bm,f}   \bigl) (\gamma_{z}(t)) \\ \notag
&\geq -e^{\frac{-2 p f}{n-1}} \left\{  \left(   \bigl(     (  \varphi'  )^{p-1}     \bigl)'-H_{\kappa,\lambda}\,(\varphi')^{p-1} \right)\circ \rho_{\bm,f}\right\}(\gamma_{z}(t)).
\end{align}
\end{lem}
\begin{proof}
Set $\Phi:=\varphi \circ \rho_{\bm,f}$.
For the function $s_{f}$ defined as (\ref{eq:radial parameter function}),
if we put $\varphi_{f}:=\varphi \circ s_{f}$,
then we have $\Phi=\varphi_{f}\circ \rho_{\bm}$.
For each $t\in (0,\tau(z))$,
the left hand side of (\ref{eq:radial p-Laplacian comparison}) can be written as
\begin{equation*}
- \bigl(     (  \varphi'_{f} )^{p-1}\bigl)'(t)+\left(\Delta_{f}\,\rho_{\bm}(\gamma_{z}(t))-\frac{2(p-1)}{n-1}\,\phi'_{f}(t)\right) \,(\varphi'_{f})^{p-1}(t).
\end{equation*}
By using (\ref{eq:Laplacian comparison}), $s_{f,z}(t)=s_{f}(t)$ and $e^{\frac{-2 f(\gamma_{z}(t))}{n-1}}=s'_{f}(t)$,
\begin{align*}
\Delta_{\frac{n+1-2p}{n-1}f,p}\, \Phi (\gamma_{z}(t)) \geq - \bigl(     (  \varphi'_{f}  )^{p-1}     \bigl)'(t)&+ H_{\kappa,\lambda}(s_{f}(t))\,s'_{f}(t)\, (\varphi'_{f})^{p-1}(t)\\
                                                                                  &-\frac{2(p-1)}{n-1}\,\phi'_{f}(t)\,(\varphi'_{f})^{p-1}(t).
\end{align*}
Notice that $(\varphi'_{f})^{p-1}(t)=(\varphi')^{p-1}(s_{f}(t))\,(s'_{f})^{p-1}(t)$,
\begin{equation*}
\bigl(     (  \varphi'_{f}  )^{p-1}     \bigl)'(t) = \left((\varphi')^{p-1}\right)'(s_{f}(t))\,(s'_{f})^{p}(t)-\frac{2(p-1)}{n-1} \phi'_{f}(t)\, (\varphi'_{f})^{p-1}(t).
\end{equation*}
These equalities tell us that
the left hand side of (\ref{eq:radial p-Laplacian comparison}) is at least
\begin{equation*}
-(s'_{f})^{p}(t)\left (   \bigl((\varphi')^{p-1}\bigl)'(s_{f}(t))- H_{\kappa,\lambda}(s_{f}(t))\,\left(\varphi'\right)^{p-1}(s_{f}(t)) \right).
\end{equation*}
Since $\rho_{\bm,f}=s_{f}\circ \rho_{\bm}$,
this is equal to the right hand side of (\ref{eq:radial p-Laplacian comparison}).
\end{proof}

We further yield the following global comparison inequality:
\begin{prop}\label{prop:global radial p-Laplacian comparison}
Let $p\in (1,\infty)$.
For $N\in (-\infty,1]$,
assume that
$(M,\bm,f)$ has lower $(\kappa,\lambda,N)$-weighted curvature bounds.
Suppose that
$f$ is $\bm$-radial.
Let $\varphi:[0,\infty)\to \mathbb{R}$ be a monotone increasing smooth function.
Then we have
\begin{equation*}
\Delta_{\frac{n+1-2p}{n-1}f,p}\,  \bigl(  \varphi \circ \rho_{\bm,f} \bigl)  \geq -e^{\frac{-2 p f}{n-1}} \left( \bigl( (\varphi' )^{p-1} \bigl)'-H_{\kappa,\lambda} \,  (\varphi' )^{p-1}\right)\circ \rho_{\bm,f}
\end{equation*}
in the following distribution sense on $M$:
For every non-negative function $\psi \in C^{\infty}_{0}(M)$ we have
\begin{align}\label{eq:global radial p-Laplacian comparison}
&\quad \, \,\int_{M}\,   \Vert \nabla \bigl( \varphi \circ \rho_{\bm,f}  \bigl)  \Vert^{p-2} \,g\bigl(\nabla \psi, \nabla \bigl(    \varphi \circ \rho_{\bm,f} \bigl)  \bigl)\, d\,m_{\frac{n+1-2p}{n-1}f}\\ \notag
&\geq -\int_{M}\,\psi\,  \left\{  \left( \bigl( (\varphi' )^{p-1} \bigl)' -H_{\kappa,\lambda}\,\left(\varphi' \right)^{p-1}\right)\circ \rho_{\bm,f}  \right\}   \,d\,m_{\frac{n+1}{n-1}f}.
\end{align}
\end{prop}
\begin{proof}
The proof is similar to the proof of Proposition \ref{prop:global finite p-Laplacian comparison}.
Similarly,
we first take a sequence $\{\Omega_{i}\}$ of closed subsets of $M$ in Lemma \ref{lem:avoiding the cut locus}.
Let $u_{i}$ be the unit outer normal vector on $\partial \Omega_{i} \setminus \partial M$ for $\Omega_{i}$.
We define $\hat{f}:= (n+1-2p)(n-1)^{-1}f$.
For the canonical volume $\vol_{i}$ on $\partial \Omega_{i}\setminus \bm$,
we put $m_{\hat{f},i}:=e^{  -\hat{f}|_{\partial \Omega_{i}\setminus \bm}}\,\vol_{i}$.
We set $\Phi:=\varphi \circ \rho_{\bm,f}$.
By integration by parts (with respect to $m_{\hat{f}}$),
and by Lemma \ref{lem:radial p-Laplacian comparison} and $g(u_{i},\nabla \rho_{\bm,f})\geq 0$,
\begin{align*}
&\quad \,\int_{\Omega_{i}}\,   \Vert \nabla \Phi   \Vert^{p-2} \,g\left(\nabla \psi, \nabla \Phi  \right)\, d\,m_{\hat{f}}\\
&\geq -\int_{\Omega_{i}}\,\psi\, e^{\frac{-2 p f}{n-1}}\,   \left\{ \left( \bigl( (\varphi' )^{p-1} \bigl)' -H_{\kappa,\lambda}\,(\varphi' )^{p-1}\right)\circ \rho_{\bm,f} \right\} \,d\,m_{\hat{f}}.
\end{align*}
Using $e^{\frac{-2 p f}{n-1}}\,m_{\hat{f}}=m_{\frac{n+1}{n-1}f}$,
we complete the proof by letting $i\to \infty$.
\end{proof}

\begin{rem}\label{rem:c1-equality case in global radial p-Laplacian comparison}
The argument in the proof of Proposition \ref{prop:global radial p-Laplacian comparison} also leads us to the following:
Under the same setting as in Proposition \ref{prop:global radial p-Laplacian comparison},
if $M$ is compact,
then the inequality (\ref{eq:global radial p-Laplacian comparison}) holds for every non-negative $\psi \in C^{1}(M)$ with $\psi|_{\bm}=0$.
\end{rem}

\section{Splitting theorems}\label{sec:Splitting theorems}

\subsection{Main splitting theorems}

Let us prove Theorem \ref{thm:splitting theorem}.
\begin{proof}[Proof of Theorem \ref{thm:splitting theorem}]
Let $\kappa \leq 0$ and $\lambda:=\sqrt{\vert \kappa \vert}$.
For $N\in (-\infty,1]$,
assume that
$(M,\bm,f)$ has lower $(\kappa,\lambda,N)$-weighted curvature bounds.
Suppose that
$f$ is bounded from above.
Let $z_{0}\in \bm$ satisfy $\tau(z_{0})=\infty$.

Let $\bm_{0}$ be the connected component of $\bm$ with $z_{0} \in \bm_{0}$.
We define a closed subset $\Omega$ of $\bm_{0}$ by
\begin{equation*}
\Omega:=\{z\in \bm_{0} \mid \tau(z)=\infty \}.
\end{equation*}

We show that
$\Omega$ is open in $\bm_{0}$.
Fix $z_{1}\in \Omega$.
Take $l>0$,
and put $x_{0}:=\gamma_{z_{1}}(l)$.
There exists an open neighborhood $U$ of $x_{0}$ contained in $\inte M \setminus \cut \bm$.
Taking $U$ smaller,
we may assume that 
for each $x \in U$
the unique foot point on $\bm$ of $x$ belongs to $\bm_{0}$.
By Lemma \ref{lem:asymptote},
there exists $\epsilon>0$ such that
for all $x \in B_{\epsilon}(x_{0})$,
all asymptotes for $\gamma_{z_{1}}$ from $x$ lie in $\inte M$.
We may assume $U\subset B_{\epsilon}(x_{0})$.
Fix $x_{1}\in U$,
and take an asymptote $\gamma_{x_{1}}:[0,\infty)\to M$ for $\gamma_{z_{1}}$ from $x_{1}$.
For $t>0$,
define a function $b_{\gamma_{z_{1}},t}:M\to \mathbb{R}$ by
\begin{equation*}
b_{\gamma_{z_{1}},t}(x):=b_{\gamma_{z_{1}}}(x_{1})+t-d_{M}(x,\gamma_{x_{1}}(t)).
\end{equation*}
We see that
$b_{\gamma_{z_{1}},t}-\rho_{\bm}$ is a support function of $b_{\gamma_{z_{1}}}-\rho_{\bm}$ at $x_{1}$.
Since $\gamma_{x_{1}}$ lie in $\inte M$,
for every $t>0$
the function $b_{\gamma_{z_{1}},t}$ is smooth on a neighborhood of $x_{1}$.
From Lemma \ref{lem:finite pointed Laplacian comparison}
we deduce
\begin{equation*}
\Delta_{f} b_{\gamma_{z_{1}},t}(x_{1})\leq -H_{\kappa}\bigl( e^{\frac{-2\sup f}{n-1}}t\bigl)\,e^{\frac{-2f(x_{1})}{n-1}},
\end{equation*}
where $H_{\kappa}$ is defined as (\ref{eq:pointed model mean curvature}).
Note that $H_{\kappa}(s)$ tends to $-(n-1)\sqrt{\vert \kappa \vert}$ as $s\to \infty$.
Furthermore,
$\rho_{\bm}$ is smooth on $U$,
and by (\ref{eq:Laplacian comparison})
we have
\begin{equation*}
\Delta_{f} \rho_{\bm}\geq (n-1)\sqrt{\vert \kappa \vert}e^{\frac{-2f}{n-1}}
\end{equation*}
on $U$.
Hence
$b_{\gamma_{z_{1}}}-\rho_{\bm}$ is $f$-subharmonic on $U$.
Now,
$b_{\gamma_{z_{1}}}-\rho_{\bm}$ takes the maximal value $0$ at $x_{1}$.
Lemma \ref{lem:maximal principle} implies $b_{\gamma_{z_{1}}}=\rho_{\bm}$ on $U$.
By Lemma \ref{lem:busemann function},
the set $\Omega$ is open in $\bm_{0}$.

The connectedness of $\bm_{0}$ leads to $\Omega=\bm_{0}$.
By Lemma \ref{lem:splitting lemma},
$\bm$ is connected and $\cut \bm = \emptyset$.
The equality in (\ref{eq:Laplacian comparison}) holds on $\inte M$.
For each $z \in \bm$,
choose an orthonormal basis $\{e_{z,i}\}_{i=1}^{n-1}$ of $T_{z}\bm$.
Let $\{Y_{z,i}\}_{i=1}^{n-1}$ be the $\bm$-Jacobi fields along $\gamma_{z}$ with $Y_{z,i}(0)=e_{z,i},\,Y'_{z,i}(0)=-A_{u_{z}}e_{z,i}$.
By Lemma \ref{lem:Equality in Laplacian comparison},
for all $i$
we see $Y_{z,i}=F_{\kappa,\lambda,z}\,E_{z,i}$ on $[0,\infty)$,
where $\{E_{z,i}\}_{i=1}^{n-1}$ are the parallel vector fields with $E_{z,i}(0)=e_{z,i}$.
Moreover,
if $N\in (-\infty,1)$,
then $f\circ \gamma_{z}$ is constant on $[0,\infty)$.
We define a diffeomorphism $\Phi:[0,\infty)\times \bm\to M$ by $\Phi(t,z):=\gamma_{z}(t)$.
The rigidity of Jacobi fields implies that
$\Phi$ is a Riemannian isometry with boundary from $[0,\infty)\times_{F_{\kappa,\lambda}} \bm$ to $M$.
We complete the proof of Theorem \ref{thm:splitting theorem}.
\end{proof}

\begin{rem}
The author \cite{Sa3} has concluded that
under the same setting as in Theorem \ref{thm:splitting theorem},
if $\kappa=0$,
then $M$ is isometric to a warped product (see Corollary 1.4 in \cite{Sa3}).
The author does not know whether
the same conclusion holds when $\kappa<0$.
\end{rem}

\subsection{Weighted Ricci curvature on the boundary}\label{sec:Weighted Ricci curvature on the boundary}
We next recall the following formula (see e.g., Lemma 5.5 in \cite{Sa3}):
\begin{lem}\label{lem:boundary weighted Ricci curvature}
Let $z\in \bm$,
and take a unit vector $v$ in $T_{z}\bm$.
Choose an orthonormal basis $\{ e_{z,i} \}_{i=1}^{n-1}$ of $T_{z}\bm$ with $e_{z,1}=v$.
Then
\begin{align}\label{eq:boundary weighted Ricci curvature}
\ric^{N-1}_{f|_{\bm}}(v)  = &\ric^{N}_{f}(v) +g((\nabla f)_{z},u_{z})\,g(S(v,v),u_{z})\\
                                        & - K_{g}(u_{z},v)+\tr A_{S(v,v)}-\sum_{i=1}^{n-1} \Vert S(v,e_{z,i})\Vert^{2}\notag
\end{align}
for all $N \in (-\infty,\infty)$,
where $K_{g}(u_{z},v)$ denotes the sectional curvature of the $2$-plane at $z$ spanned by $u_{z}$ and $v$.
\end{lem}

Using Lemma \ref{lem:boundary weighted Ricci curvature},
we show the following:
\begin{lem}\label{lem:warped product boundary weighted Ricci curvature}
Take $z\in \bm$,
and take a unit vector $v$ in $T_{z}\bm$.
If $M$ is isometric to $[0,\infty)\times_{F_{\kappa,\lambda}}\bm$,
then for all $N\in (-\infty,\infty)$
we have
\begin{align*}\label{eq:warped product boundary weighted Ricci curvature}
\ric^{N-1}_{f|_{\bm}}(v) = \ric^{N}_{f}(v)&+(n-1)\lambda^{2} e^{\frac{-4f(z)}{n-1}}-\kappa \, e^{\frac{-4f(z)}{n-1}}\\
                                                             &- \lambda \, g((\nabla f)_{z},u_{z})\, e^{\frac{-2f(z)}{n-1}}+\frac{\Hess f(u_{z},u_{z})}{n-1}.
\end{align*}
\end{lem}
\begin{proof}
We choose an orthonormal basis $\{e_{z,i}\}_{i=1}^{n-1}$ of $T_{z}\bm$ with $e_{z,1}=v$.
Let $\{Y_{z,i}\}^{n-1}_{i=1}$ denote the $\bm$-Jacobi fields along $\gamma_{z}$ with $Y_{z,i}(0)=e_{z,i},\,Y'_{z,i}(0)=-A_{u_{z}}e_{z,i}$.
By the rigidity assumption,
$Y_{z,i}=F_{\kappa,\lambda,z}\,E_{z,i}$,
where $\{E_{z,i}\}^{n-1}_{i=1}$ are the parallel vector fields with $E_{z,i}(0)=e_{z,i}$.
Therefore,
for all $i$
it holds that
\begin{equation}\label{eq:warped product shape operator}
A_{u_{z}}e_{z,i}=-Y'_{z,i}(0)=-\left(\frac{g((\nabla f)_{z},u_{z})}{n-1}-\lambda e^{\frac{-2f(z)}{n-1}}    \right)e_{z,i}.
\end{equation}
From (\ref{eq:warped product shape operator})
we deduce $S(v,e_{z,i})=0_{z}$ for all $i \neq 1$,
and we also deduce
\begin{align}\label{eq:warped product second fundamental form}
S(v,v)            &= -\left(\frac{g((\nabla f)_{z},u_{z})}{n-1}-\lambda e^{\frac{-2f(z)}{n-1}}    \right)u_{z},\\ \label{eq:warped product trace shape operator}
\tr A_{S(v,v)} &=(n-1)\left(\frac{g((\nabla f)_{z},u_{z})}{n-1}-\lambda e^{\frac{-2f(z)}{n-1}}    \right)^{2} .
\end{align}
The sectional curvature $K_{g}(u_{z},v)$ is equal to $-g(Y''_{z,1}(0),v)$,
and hence
\begin{equation}\label{eq:warped product sectional curvature}
K_{g}(u_{z},v)= -\frac{\Hess f(u_{z},u_{z})}{n-1}-\left( \frac{g((\nabla f)_{z},u_{z})}{n-1}  \right)^{2}+\kappa  e^{\frac{-4f(z)}{n-1}}.
\end{equation}
Lemma \ref{lem:boundary weighted Ricci curvature} together with (\ref{eq:warped product second fundamental form}), (\ref{eq:warped product trace shape operator}), (\ref{eq:warped product sectional curvature})
yields the desired one.
\end{proof}

\subsection{Multi-splitting}\label{sec:Multi-splitting}
On a connected complete Riemannian manifold $M_{0}$ (without boundary),
a minimal geodesic $\gamma:\mathbb{R} \to M_{0}$ is said to be a \textit{line}.
Wylie \cite{W} has proved the following splitting theorem of Cheeger-Gromoll type (see Theorem 1.2 and Corollary 1.3 in \cite{W}):
\begin{thm}[\cite{W}]\label{thm:splitting theorem of Cheeger-Gromoll type}
Let $M_{0}$ be a connected complete Riemannian manifold,
and let $f_{0}:M_{0}\to \mathbb{R}$ be a smooth function bounded from above.
For $N \in (-\infty,1]$,
suppose $\ric^{N}_{f_{0},M_{0}}\geq 0$.
If $M_{0}$ contains a line,
then there exists a Riemannian manifold $\widetilde{M}_{0}$ such that
$M_{0}$ is isometric to a warped product space over $\mathbb{R}\times \widetilde{M}_{0}$;
moreover,
if $N \in (-\infty,1)$,
then $M_{0}$ is isometric to the standard product $\mathbb{R}\times \widetilde{M}_{0}$.
\end{thm}

From Theorem \ref{thm:splitting theorem of Cheeger-Gromoll type}
we derive the following corollary of Theorem \ref{thm:splitting theorem}:
\begin{cor}\label{cor:boundary splitting}
Let $\kappa \leq 0$ and $\lambda:=\sqrt{\vert \kappa \vert}$.
For $N \in (-\infty,1)$,
assume that
$(M,\bm,f)$ has lower $(\kappa,\lambda,N)$-weighted curvature bounds.
Suppose that
$f$ is bounded from above.
If for some $z_{0}\in \bm$
we have $\tau(z_{0})=\infty$,
then there exist an integer $k\in \{0,\dots,n-1\}$ and
an $(n-1-k)$-dimensional Riemannian manifold $\widetilde{\bm}$ containing no line such that
$\bm$ is isometric to $\mathbb{R}^{k}\times \widetilde{\bm}$;
in particular,
$M$ is isometric to $[0,\infty)\times_{F_{\kappa,\lambda}}\bigl(  \mathbb{R}^{k}\times \widetilde{\bm} \bigl)$.
\end{cor}
\begin{proof}
Due to Theorem \ref{thm:splitting theorem},
$M$ is isometric to $[0,\infty)\times_{F_{\kappa,\lambda}} \bm$,
and for each $z\in \bm$
the function $f\circ \gamma_{z}$ is constant on $[0,\infty)$.
In particular,
$g((\nabla f)_{z},u_{z})=0$ and $\Hess f(u_{z},u_{z})=0$.
By Lemma \ref{lem:warped product boundary weighted Ricci curvature},
and by $\kappa \leq 0$ and $\lambda=\sqrt{\vert \kappa \vert}$,
for every unit vector $v$ in $T_{z}\bm$
we have
\begin{align*}
\ric^{N-1}_{f|_{\bm}}(v) &    =    \ric^{N}_{f}(v)+(n-1)\lambda^{2} e^{\frac{-4f(z)}{n-1}}-\kappa e^{\frac{-4f(z)}{n-1}}\\
                                     & \geq  \ric^{N}_{f}(v)+(n-1)\lambda^{2} e^{\frac{-4f(z)}{n-1}}\\
                                     & \geq  (n-1)\kappa\, e^{\frac{-4f(z)}{n-1}}+(n-1)\lambda^{2} e^{\frac{-4f(z)}{n-1}}=0.
\end{align*}
It follows that $\ric^{N-1}_{f|_{\bm},\bm} \geq 0$.
Now,
$N-1$ is smaller than $1$,
and $f|_{\bm}$ is bounded from above.
Therefore,
by applying Theorem \ref{thm:splitting theorem of Cheeger-Gromoll type} to $\bm$ inductively,
we complete the proof.
\end{proof}

\subsection{Variants of splitting theorems}\label{sec:Variants of splitting theorems}
We study generalizations of rigidity results of Kasue \cite{K2}, Croke and Kleiner \cite{CK} and Ichida \cite{I}
for manifolds with boundary whose boundaries are disconnected.

Wylie \cite{W} has proved the following (see Theorem 5.1 in \cite{W}):
\begin{thm}[\cite{W}]\label{thm:disconnected splitting1}
For $N \in (-\infty,1]$,
assume that
$(M,\bm,f)$ has lower $(0,0,N)$-weighted curvature bounds.
Let $\bm$ be disconnected,
and let $\{\bm_{i}\}_{i=1,2,\dots}$ denote the connected components of $\bm$.
Let $\bm_{1}$ be compact,
and put $D:=\inf_{i=2,3,\dots}\, d_{M}(\bm_{1},\bm_{i})$.
Then $M$ is isometric to $[0,D]\times_{F_{0,0}} \bm_{1}$,
and $\ric^{N}_{f}(\gamma'_{z}(t))=0$ for all $z \in \bm_{1}$ and $t\in [0,D]$.
\end{thm}

For $\kappa>0$ and $\lambda<0$,
put $D_{\kappa,\lambda}:=\inf\, \{\,t>0\,|\, s'_{\kappa,\lambda}(t)=0\,\}$.
By using Theorem \ref{thm:disconnected splitting1},
we obtain the following splitting:
\begin{thm}\label{thm:disconnected splitting2}
Let $\kappa>0$.
For $N\in (-\infty,1]$,
assume that
$(M,\bm,f)$ has lower $(\kappa,\lambda,N)$-weighted curvature bounds.
Let $\bm$ be disconnected,
and let $\{\bm_{i}\}_{i=1,2,\dots}$ denote the connected components of $\bm$.
Let $\bm_{1}$ be compact,
and put $D:=\inf_{i=2,3,\dots}\, d_{M}(\bm_{1},\bm_{i})$.
Suppose additionally that
there exists $\delta \in \mathbb{R}$ such that $f\leq (n-1)\delta$ on $M$.
Then
\begin{equation*}
\lambda<0,\quad D\leq 2e^{2\delta}D_{\kappa,\lambda}.
\end{equation*}
Moreover,
if $D=2e^{2\delta}D_{\kappa,\lambda}$,
then $M$ is isometric to $[0,D]\times_{F_{\kappa,\lambda}} \bm_{1}$,
and $f=(n-1)\delta$ on $M$.
\end{thm}
\begin{proof}
If we have $\lambda \geq 0$,
then Theorem \ref{thm:disconnected splitting1} tells us that
$M$ is isometric to $[0,D]\times_{F_{0,0}} \bm_{1}$,
and $\ric^{N}_{f}(\gamma'_{z}(t))=0$ for all $z \in \bm_{1},\,t\in [0,D]$.
This contradicts $\kappa>0$,
and hence $\lambda<0$.

Let us prove that
if $D\geq 2e^{2\delta}D_{\kappa,\lambda}$,
then $M$ is isometric to the twisted product $[0,2e^{2\delta}D_{\kappa,\lambda}]\times_{F_{\kappa,\lambda}} \bm_{1}$,
and $f=(n-1)\delta$ on $M$.
Suppose $D\geq 2e^{2\delta}D_{\kappa,\lambda}$.
There exists a connected component $\bm_{2}$ of $\bm$ such that $d_{M}(\bm_{1},\bm_{2})=D$ (cf. Lemma 1.6 in \cite{K2}).
For each $i=1,2$,
let $\rho_{\bm_{i}}:M\to \mathbb{R}$ be the function defined by $\rho_{\bm_{i}}(x):=d_{M}(x,\bm_{i})$.
Set
\begin{equation*}
\Omega:=\{x\in \inte M \mid \rho_{\bm_{1}}(x)+\rho_{\bm_{2}}(x)=D\}.
\end{equation*}

We show that
$\Omega$ is open in $\inte M$.
Fix $x\in \Omega$.
For each $i=1,2$,
we take a foot point $z_{x,i}\in \bm_{i}$ on $\bm_{i}$ of $x$ such that $d_{M}(x,z_{x,i})=\rho_{\bm_{i}}(x)$.
From the triangle inequality
we derive $d_{M}(z_{x,1},z_{x,2})=D$.
The minimal geodesic $\gamma:[0,D]\to M$ from $z_{x,1}$ to $z_{x,2}$ is orthogonal to $\bm$ at $z_{x,1}$ and at $z_{x,2}$.
Furthermore,
$\gamma|_{(0,D)}$ lies in $\inte M$ and passes through $x$.
There exists an open neighborhood $U$ of $x$ such that $\rho_{\bm_{i}}$ is smooth on $U$.
In view of (\ref{eq:finite Laplacian comparison}),
for all $y\in U$,
we see
\begin{align}\label{eq:disconnected}
&\quad \,\,-\frac{\Delta_{f}\, \left(\rho_{\bm_{1}}+\rho_{\bm_{2}}\right)(y)}{(n-1)e^{\frac{-2f(y)}{n-1}}}
   \leq \frac{s'_{\kappa,\lambda}}{s_{\kappa,\lambda}}(\rho_{\bm_{1},\delta}(y))+ \frac{s'_{\kappa,\lambda}}{s_{\kappa,\lambda}}(\rho_{\bm_{2},\delta}(y))\\ \notag
&  =    \frac{s'_{\kappa,\lambda}(\rho_{\bm_{1},\delta}(y)+\rho_{\bm_{2},\delta}(y))-\lambda s_{\kappa,\lambda}(\rho_{\bm_{1},\delta}(y)+\rho_{\bm_{2},\delta}(y))}{s_{\kappa,\lambda}(\rho_{\bm_{1},\delta}(y))  s_{\kappa,\lambda}(\rho_{\bm_{2},\delta}(y))},
\end{align}
where $\rho_{\bm_{i},\delta}:=e^{-2\delta}\rho_{\bm_{i}}$.
Since $\kappa>0$,
the function $s'_{\kappa,\lambda}/s_{\kappa,\lambda}$ is monotone decreasing on $(0,\const)$,
and satisfies $s'_{\kappa,\lambda}(2D_{\kappa,\lambda})/s_{\kappa,\lambda}(2D_{\kappa,\lambda})=\lambda$.
By $D\geq 2e^{2\delta}D_{\kappa,\lambda}$ and the triangle inequality,
$\rho_{\bm_{1},\delta}+\rho_{\bm_{2},\delta}\geq 2D_{\kappa,\lambda}$ on $U$.
The inequality (\ref{eq:disconnected}) tells us that
$-(\rho_{\bm_{1}}+\rho_{\bm_{2}})$ is $f$-subharmonic on $U$.
By Lemma \ref{lem:maximal principle},
$\Omega$ is open in $\inte M$.

The connectedness of $\inte M$ implies $\inte M=\Omega$.
The equality in (\ref{eq:finite Laplacian comparison}) holds.
For each $z\in \bm_{1}$,
choose an orthonormal basis $\{e_{z,i}\}_{i=1}^{n-1}$ of $T_{z}\bm$.
Let $\{Y_{z,i}\}_{i=1}^{n-1}$ be the $\bm$-Jacobi fields along $\gamma_{z}$ with $Y_{z,i}(0)=e_{z,i},\,Y'_{z,i}(0)=-A_{u_{z}}e_{z,i}$.
For all $i$
we see $Y_{z,i}=F_{\kappa,\lambda,z} E_{z,i}$ on $[0,D]$,
where $\{E_{z,i}\}_{i=1}^{n-1}$ are the parallel vector fields with $E_{z,i}(0)=e_{z,i}$.
Moreover,
$f\circ \gamma_{z}=(n-1)\delta$ on $[0,D]$ (see Remark \ref{rem:equality in finite Laplacian comparison}).
We see $D=2e^{2\delta}D_{\kappa,\lambda}$.
By the rigidity of Jacobi fields,
a map $\Phi:[0,D]\times \bm_{1}\to M$ defined by $\Phi(t,z):=\gamma_{z}(t)$ is a desired Riemannian isometry with boundary.
\end{proof}

\section{Inscribed radii}\label{sec:Inscribed radii}
We denote by $L_{g_{f}},\,d^{g_{f}}_{M},\,\rho^{g_{f}}_{\bm}$ and $\IR_{g_{f}} M$
the length,
the Riemannian distance,
the distance function from the boundary and the inscribed radius on $M$ induced from the Riemannian metric $g_{f}:=e^{\frac{-4f}{n-1}}g$.

\subsection{Inscribed radius comparisons}\label{sec:Inscribed radius comparisons}
We first show the following:
\begin{lem}\label{lem:Inscribed radius comparison}
Let $\kappa$ and $\lambda$ satisfy the ball-condition.
For $N \in (-\infty,1]$,
assume that
$(M,\bm,f)$ has lower $(\kappa,\lambda,N)$-weighted curvature bounds.
Then we have $\IR_{g_{f}} M\leq \const$.
\end{lem}
\begin{proof}
Take $x\in M$,
and a foot point $z_{x}$ on $\bm$ of $x$.
Then we have
\begin{equation*}
\rho^{g_{f}}_{\bm}(x) \leq L_{g_{f}}(\gamma_{z_{x}}|_{[0,l]}) = \int^{l}_{0}\, e^{\frac{-2f(\gamma_{z_{x}}(a))}{n-1}}\,da \leq \tau_{f}(z_{x})\leq \sup_{z\in \bm}\,\tau_{f}(z),
\end{equation*}
where $l:=\rho_{\bm}(x)$.
Lemma \ref{lem:Cut point comparisons} implies the desired inequality.
\end{proof}

From Lemma \ref{lem:Cut point comparisons} and $\IR M=\sup_{z\in \bm}\tau(z)$,
we also derive:
\begin{lem}\label{lem:finite Inscribed radius comparison}
Let $\kappa$ and $\lambda$ satisfy the ball-condition.
For $N\in (-\infty,1]$,
assume that
$(M,\bm,f)$ has lower $(\kappa,\lambda,N)$-weighted curvature bounds.
Suppose additionally that
there is $\delta \in \mathbb{R}$ such that $f\leq (n-1)\delta$ on $M$.
Then we have $\IR M \leq C_{\kappa\,e^{-4\delta,\lambda\,e^{-2\delta}}}$.
\end{lem}

\subsection{Inscribed radius rigidity}
We now prove Theorem \ref{thm:inscribed radius rigidity}.
\begin{proof}[Proof of Theorem \ref{thm:inscribed radius rigidity}]
Let $\kappa$ and $\lambda$ satisfy the ball-condition.
For $N \in (-\infty,1]$,
assume that
$(M,\bm,f)$ has lower $(\kappa,\lambda,N)$-weighted curvature bounds.
By Lemma \ref{lem:Inscribed radius comparison},
we have (\ref{eq:inscribed radius rigidity}).
Let $x_{0}\in M$ satisfy $\rho^{g_{f}}_{\bm}(x_{0})=\const$,
which will become the center of $M$.
Put $l:=\rho_{\bm}(x_{0})$,
\begin{equation*}
\Omega:=\{x \in \inte M \setminus \{x_{0}\} \,|\, \rho_{\bm}(x)+\rho_{x_{0}}(x)=l\}.
\end{equation*}

We show that
$\Omega$ is open in $\inte M \setminus \{x_{0}\}$.
Fix $x\in \Omega$,
and take a foot point $z_{x}$ on $\bm$ of $x$.
Note that
$z_{x}$ is also a foot point on $\bm$ of $x_{0}$.
Let $\gamma:[0,l]\to M$ be the minimal geodesic from $z_{x}$ to $x_{0}$.
Then $\gamma|_{(0,l)}$ passes through $x$.
There exists an open neighborhood $U$ of $x$ such that the distance functions $\rho_{x_{0}}$ and $\rho_{\bm}$ are smooth on $U$,
and for every $y\in U$
there exists a unique minimal geodesic in $M$ from $x_{0}$ to $y$ that lies in $\inte M$.
By Lemma \ref{lem:pointed Laplacian comparison} and (\ref{eq:Laplacian comparison}),
for each $y \in U$
we have
\begin{align}\label{eq:combination of Laplacian comparison}
-\frac{\Delta_{f} (\rho_{\bm}+\rho_{x_{0}})(y)}{(n-1)\,e^{\frac{-2 f(y)}{n-1}}} & \leq      \frac{s'_{\kappa,\lambda}}{s_{\kappa,\lambda}}(s_{f,z_{y}}(\rho_{\bm}(y)))+\frac{s'_{\kappa}}{s_{\kappa}}(s_{f,v_{y}}(\rho_{x_{0}}(y)))\\ \notag
                                                                              &     =   \frac{s_{\kappa,\lambda}(s_{f,z_{y}}(\rho_{\bm}(y))+s_{f,v_{y}}(\rho_{x_{0}}(y)))}{s_{\kappa,\lambda}(s_{f,z_{y}}(\rho_{\bm}(y)))\, s_{\kappa}(s_{f,v_{y}}(\rho_{x_{0}}(y)))},
\end{align}
where $z_{y}$ is a unique foot point on $\bm$ of $y$,
and $v_{y}$ is the initial velocity vector of the unique minimal geodesic from $x_{0}$ to $y$.
Let us define $\rho^{g_{f}}_{x_{0}}:=d^{g_{f}}_{M}(\cdot,x_{0})$.
The triangle inequality for $d^{g_{f}}_{M}$ leads us to
\begin{align}\label{eq:weighted triangle inequality}
s_{f,z_{y}}(\rho_{\bm}(y))+s_{f,v_{y}}(\rho_{x_{0}}(y))&=L_{g_{f}}(\gamma_{z_{y}}|_{[0,\rho_{\bm}(y)]})+L_{g_{f}}(\gamma_{v_{y}}|_{[0,\rho_{x_{0}}(y)]}) \\ \notag
                                                                                   &\geq \rho^{g_{f}}_{\bm}(y)+\rho^{g_{f}}_{x_{0}}(y)  \geq \rho^{g_{f}}_{\bm}(x_{0})=\const.
\end{align}
By (\ref{eq:combination of Laplacian comparison}) and (\ref{eq:weighted triangle inequality}),
we have $\Delta_{f} (\rho_{\bm}+\rho_{x_{0}})(y)\geq 0$.
Lemma \ref{lem:maximal principle} tells us that $U\subset \Omega$,
and
$\Omega$ is open.

Since $\inte M\setminus \{x_{0}\}$ is connected, 
we have $\Omega=\inte M\setminus \{x_{0}\}$,
and hence $\rho_{\bm}+\rho_{x_{0}}=l$ on $M$.
This implies $M=B_{l}(x_{0})$.
For each $v \in U_{x_{0}}M$,
we have $\tau_{x_{0}}(v)=l$,
and $\gamma_{v}$ is orthogonal to $\bm$ at $l$.
The equality in (\ref{eq:pointed Laplacian comparison}) holds on $\inte M\setminus \{x_{0}\}$.
Choose an orthonormal basis $\{e_{v,i}\}_{i=1}^{n}$ of $T_{x_{0}}M$ with $e_{v,n}=v$.
Let $\{Y_{v,i}\}^{n-1}_{i=1}$ be the Jacobi fields along $\gamma_{v}$ with $Y_{v,i}(0)=0_{x_{0}},\,Y_{v,i}'(0)=e_{v,i}$.
By Lemma \ref{lem:equality in pointed Laplacian comparison},
for all $i$
we have $Y_{v,i}=F_{\kappa,v} \,E_{v,i}$ on $[0,l]$,
where $\{E_{v,i}\}^{n-1}_{i=1}$ are the parallel vector fields with $E_{v,i}(0)=e_{v,i}$;
moreover,
if $N \in (-\infty,1)$,
then $f\circ \gamma_{v}$ is constant on $[0,l]$.
Since the equalities in (\ref{eq:weighted triangle inequality}) hold,
we have $s_{f,v}(l)=\const$ and $F_{\kappa,v}(l)>0$;
in particular,
we have no conjugate point of $x_{0}$ along $\gamma_{v}$.
Thus,
a map $\Phi:[0,l]\times U_{x_{0}}M\to M$ defined by $\Phi(t,v):=\gamma_{v}(t)$ is a Riemannian isometry with boundary from $[0,l]\times_{F_{\kappa}} \mathbb{S}^{n-1}$ to $M$.

Assume $N \in (-\infty,1)$.
Then $f=(n-1)\delta$ for some constant $\delta \in \mathbb{R}$;
in particular,
$F_{\kappa,v}(t)=e^{2\delta}\,s_{\kappa}\bigl(e^{-2\delta}t\bigl)=s_{\kappa e^{-4\delta}}(t)$ and $l=e^{2\delta} \const =C_{\kappa e^{-4\delta},\lambda e^{-2\delta}}$.
Hence
$[0,l]\times_{F_{\kappa}} \mathbb{S}^{n-1}$ can be written as $B^{n}_{\kappa e^{-4\delta},\lambda e^{-2\delta}}$.
This completes the proof of Theorem \ref{thm:inscribed radius rigidity}.
\end{proof}

\begin{rem}
From the argument discussed in the proof of Proposition 4.15 in \cite{WY},
one can also conclude the following:
Under the same setting as in Theorem \ref{thm:inscribed radius rigidity},
if $\rho^{g_{f}}_{\bm}(x_{0})=\const$ for some $x_{0}\in M$,
then we have $\nabla f=g(\nabla f,\nabla \rho_{x_{0}})\,\nabla \rho_{x_{0}}$ on $M$;
in particular,
$M$ is a warped product.
\end{rem}

If $f$ is bounded from above,
then we obtain the following:
\begin{thm}\label{thm:finite inscribed radius rigidity}
Let us assume that
$\kappa$ and $\lambda$ satisfy the ball-condition.
For $N \in (-\infty,1]$,
assume that
$(M,\bm,f)$ has lower $(\kappa,\lambda,N)$-weighted curvature bounds.
Suppose additionally that
there exists $\delta \in \mathbb{R}$ such that $f\leq (n-1)\delta$ on $M$.
Then we have
\begin{equation}\label{eq:finite inscribed radius rigidity}
\IR M \leq C_{\kappa\,e^{-4\delta},\lambda\,e^{-2\delta}}.
\end{equation}
If $\rho_{\bm}(x_{0})=C_{\kappa\,e^{-4\delta},\lambda\,e^{-2\delta}}$ for some $x_{0}\in M$,
then $M$ is isometric to $B^{n}_{\kappa e^{-4\delta},\lambda e^{-2\delta}}$,
and $f=(n-1)\delta$ on $M$.
\end{thm}
\begin{proof}
The inequality (\ref{eq:finite inscribed radius rigidity}) follows from Lemma \ref{lem:finite Inscribed radius comparison}.
Let $x_{0}$ satisfy $\rho_{\bm}(x_{0})=C_{\kappa\,e^{-4\delta},\lambda\,e^{-2\delta}}$,
which will become the center.
Put $l:=\rho_{\bm}(x_{0})$,
\begin{equation*}\label{eq:equality domain}
\Omega:=\{x\in \inte M \setminus \{x_{0}\} \,|\, \rho_{\bm}(x)+\rho_{x_{0}}(x)=l\}.
\end{equation*}
We prove that
$\Omega$ is open in $\inte M \setminus \{x_{0}\}$.
For a fixed point $x \in \Omega$,
there exists an open neighborhood $U$ of $x$ such that $\rho_{x_{0}}$ and $\rho_{\bm}$ are smooth on $U$,
and for every $y \in U$
there exists a unique minimal geodesic in $M$ from $x_{0}$ to $y$ that lies in $\inte M$.
Let us define $\rho_{\bm,\delta}:=e^{-2\delta}\rho_{\bm}$ and $\rho_{x_{0},\delta}:=e^{-2\delta}\rho_{x_{0}}$.
By Lemma \ref{lem:finite pointed Laplacian comparison} and (\ref{eq:finite Laplacian comparison}),
for each $y\in U$
\begin{align*}
-\frac{\Delta_{f} (\rho_{\bm}+\rho_{x_{0}})(y)}{(n-1)\,e^{\frac{-2 f(y)}{n-1}}} & \leq     \frac{s'_{\kappa,\lambda}}{s_{\kappa,\lambda}}(\rho_{\bm,\delta}(y))+\frac{s'_{\kappa}}{s_{\kappa}}(\rho_{x_{0},\delta}(y))  \\
                                                                              &     =   \frac{s_{\kappa,\lambda}(\rho_{\bm,\delta}(y)+\rho_{x_{0},\delta}(y))}{s_{\kappa,\lambda}(\rho_{\bm,\delta}(y))\, s_{\kappa}(\rho_{x_{0},\delta}(y))}\leq 0.
\end{align*}
Lemma \ref{lem:maximal principle} implies $U\subset \Omega$.
Hence
$\Omega$ is open.

From the connectedness of $\inte M\setminus \{x_{0}\}$
we deduce $\Omega=\inte M\setminus \{x_{0}\}$,
and hence $\rho_{\bm}+\rho_{x_{0}}=l$ on $M$.
This implies $M=B_{l}(x_{0})$.
For each $v \in U_{x_{0}}M$,
we have $\tau_{x_{0}}(v)=l$,
and $\gamma_{v}$ is orthogonal to $\bm$ at $l$.
The equality in (\ref{eq:pointed Laplacian comparison}) holds on $\inte M\setminus \{x_{0}\}$.
Choose an orthonormal basis $\{e_{v,i}\}_{i=1}^{n}$ of $T_{x_{0}}M$ with $e_{v,n}=v$.
Let $\{Y_{v,i}\}^{n-1}_{i=1}$ be the Jacobi fields along $\gamma_{v}$ with $Y_{v,i}(0)=0_{x_{0}},\,Y_{v,i}'(0)=e_{v,i}$.
For all $i$
we have $Y_{v,i}=F_{\kappa,v} \,E_{v,i}$ on $[0,l]$,
where $F_{\kappa,v}$ is defined as (\ref{eq:pointed model Jacobi field}),
and $\{E_{v,i}\}^{n-1}_{i=1}$ are the parallel vector fields with $E_{v,i}(0)=e_{v,i}$;
moreover,
$f \circ \gamma_{v}=(n-1)\delta$ on $[0,l]$ (see Remark \ref{rem:equality in finite pointed Laplacian comparison} and Lemma \ref{lem:equality in pointed Laplacian comparison}).
We see $F_{\kappa,v}=s_{\kappa e^{-4\delta}}$ on $[0,l]$.
Therefore,
a map $\Phi:[0,l]\times U_{x_{0}}M\to M$ defined by $\Phi(t,v):=\gamma_{v}(t)$ gives a desired Riemannian isometry with boundary.
\end{proof}
\section{Volume growths}\label{sec:Volume growths}

\subsection{Volume elements}\label{sec:volume elements}
We first recall that
$\tau_{f}$ is defined as (\ref{eq:boundary cut value}).
For $z\in \bm$ and $s \in (0,\tau_{f}(z))$
we define
\begin{equation}\label{eq:volume element}
\hat{\theta}_{f}(s,z):=\theta_{f}(t_{f,z}(s),z),
\end{equation}
where $\theta_{f}(t,z)$ is defined as (\ref{eq:volume element}),
and $t_{f,z}$ is the inverse function of the function $s_{f,z}$ defined as (\ref{eq:speed changing function}).

We show the following volume element comparison inequality:
\begin{lem}\label{lem:volume element comparison}
Let $z\in \bm$.
For $N \in (-\infty,1]$,
let us assume that
$\ric^{N}_{f}(\gamma'_{z}(t))\geq (n-1)\kappa e^{\frac{-4 f(\gamma_{z}(t))}{n-1}}$ for all $t \in (0,\tau(z))$,
and $H_{f,z}\geq (n-1)\lambda e^{\frac{-2 f(z)}{n-1}}$.
Then for all $s_{1},s_{2} \in [0,\tau_{f}(z))$ with $s_{1}\leq s_{2}$
\begin{equation*}
\frac{\hat{\theta}_{f}(s_{2},z)}{ \hat{\theta}_{f}(s_{1},z)}\leq \frac{s^{n-1}_{\kappa,\lambda}(s_{2})}{s^{n-1}_{\kappa,\lambda}(s_{1})}.
\end{equation*}
In particular,
for all $s\in [0,\tau_{f}(z))$
\begin{equation}\label{eq:absolute volume element comparison}
\hat{\theta}_{f}(s,z)\leq e^{-f(z)}\,s^{n-1}_{\kappa,\lambda}(s).
\end{equation}
\end{lem}
\begin{proof}
By (\ref{eq:Laplacian representation}) and (\ref{eq:Basic comparison}),
for all $s \in (0,\tau_{f}(z))$
we see
\begin{equation*}
\frac{d}{ds}\log \frac{\hat{\theta}_{f}(s,z)}{s^{n-1}_{\kappa,\lambda}(s)}=-\bigl(e^{\frac{2f}{n-1}}\,\Delta_{f}\rho_{\bm}\bigl)(\gamma_{z}(t_{f,z}(s)))+H_{\kappa,\lambda}(s)\leq 0,
\end{equation*}
where $H_{\kappa,\lambda}$ is defined as (\ref{eq:model mean curvature}).
This implies the lemma.
\end{proof}
\begin{rem}\label{rem:equality in volume element comparison}
Assume that
for some $s_{0}\in (0,\tau_{f}(z))$
the equality in (\ref{eq:absolute volume element comparison}) holds.
Then the equality in (\ref{eq:absolute volume element comparison}) holds on $[0,s_{0}]$;
in particular,
the equality in (\ref{eq:Basic comparison}) holds on $[0,s_{0}]$ (see Lemma \ref{lem:Equality in Basic comparison}).
\end{rem}

If $f$ is bounded from above,
then we have the following:
\begin{lem}\label{lem:finite volume element comparison}
Let $z\in \bm$.
Let $\kappa$ and $\lambda$ satisfy the monotone-condition.
For $N \in (-\infty,1]$,
assume that
$\ric^{N}_{f}(\gamma'_{z}(t))\geq (n-1)\kappa\,e^{\frac{-4 f(\gamma_{z}(t))}{n-1}}$ for all $t \in (0,\tau(z))$,
and $H_{f,z}\geq (n-1)\lambda e^{\frac{-2 f(z)}{n-1}}$.
We suppose additionally that
there exists $\delta \in \mathbb{R}$ such that $f \circ \gamma_{z} \leq (n-1)\delta$ on $(0,\tau(z))$.
Then for all $t_{1},t_{2} \in [0,\tau(z))$ with $t_{1}\leq t_{2}$
we have
\begin{equation*}\label{eq:Jacobi comparison}
\frac{\theta_{f}(t_{2},z)}{ \theta_{f}(t_{1},z)}\leq \frac{s^{n-1}_{\kappa\,e^{-4\delta},\lambda\,e^{-2\delta}}(t_{2})}{s^{n-1}_{\kappa\,e^{-4\delta},\lambda\,e^{-2\delta}}(t_{1})}.
\end{equation*}
In particular,
for all $t\in [0,\tau(z))$
we have
\begin{equation}\label{eq:finite absolute volume element comparison}
\theta_{f}(t,z)\leq e^{-f(z)}\,s^{n-1}_{\kappa\,e^{-4\delta},\lambda\,e^{-2\delta}}(t).
\end{equation}
\end{lem}
\begin{proof}
From (\ref{eq:Laplacian representation}) and (\ref{eq:strictly finite Laplacian comparison})
we deduce
\begin{equation*}
\frac{d}{dt}\log \frac{\theta_{f}(t,z)}{s^{n-1}_{\kappa,\lambda}(e^{-2\delta}\,t)}=-\Delta_{f}\rho_{\bm}(\gamma_{z}(t))+H_{\kappa,\lambda}(e^{-2\delta}t) \,e^{-2\delta}\leq 0.
\end{equation*}
Since $s_{\kappa\,e^{-4\delta},\lambda\,e^{-2\delta}}(t)=s_{\kappa,\lambda}(e^{-2\delta}\,t)$,
we obtain the desired inequality.
\end{proof}
\begin{rem}\label{rem:equality in finite volume element comparison}
Assume that
for some $t_{0}\in (0,\tau(z))$
the equality in (\ref{eq:finite absolute volume element comparison}) holds.
Then the equality in (\ref{eq:finite absolute volume element comparison}) holds on $[0,t_{0}]$;
in particular,
the equality in (\ref{eq:strictly finite Laplacian comparison}) holds on $[0,t_{0}]$ (see Remark \ref{rem:equality in strictly finite Laplacian comparison}).
\end{rem}

\subsection{Absolute comparisons}\label{sec:Absolute volume comparisons}
We define $\check{\theta}_{f}:[0,\infty)\times \bm \to \mathbb{R}$ by
\begin{equation}\label{eq:extended volume element}
\check{\theta}_{f}(s,z):=\begin{cases}
                                                    \hat{\theta}_{f}(s,z) & \text{if $s<\tau_{f}(z)$}, \\
                                                                                0           & \text{if $s \geq \tau_{f}(z)$}.
                                                   \end{cases}
\end{equation}

To prove our volume comparison theorems,
we need the following:
\begin{lem}\label{lem:integration formula}
Let $\bm$ be compact.
Then for all $r>0$
\begin{equation*}\label{eq:integration formula}
m_{\frac{n+1}{n-1}f}\left(B^{f}_{r}(\bm)\right)=\int_{\bm}\,\int^{r}_{0}\,\check{\theta}_{f}(s,z)\,ds\, d\vol_{h}.
\end{equation*}
where $B^{f}_{r}(\bm)$ is defined as $(\ref{eq:twisted inscribed radius})$,
and $\vol_{h}$ is the Riemannian volume measure on $\bm$ determined by the induced metric $h$.
\end{lem}
\begin{proof}
We give an outline of the proof.
For $r>0$,
we set
\begin{align*}
&U^{f}_{r}:=\left\{z \in \bm  \mid \tau_{f}(z)\leq r \right\},  \quad \widehat{U}^{f}_{r}:=\bigcup_{z\in U^{f}_{r}} \left\{ \gamma_{z}(t)\mid t\in [0,\tau(z)) \right\}, \\
&V^{f}_{r}:=\left\{z \in \bm  \mid \tau_{f}(z)> r \right\},      \quad \widehat{V}^{f}_{r}:=\bigcup_{z\in V^{f}_{r}} \left\{ \gamma_{z}(t)\mid t\in [0,t_{f,z}(r)] \right\}.
\end{align*}
For all $z \in \bm$ and $t\in [0,\tau(z))$
we see $\rho_{\bm,f}(\gamma_{z}(t))=s_{f,z}(t)$.
Hence,
by the straightforward argument,
one can verify $B^{f}_{r}(\bm)\setminus \cut \bm=\widehat{U}^{f}_{r} \sqcup \widehat{V}^{f}_{r}$.
In virtue of the coarea formula and the Fubini theorem,
\begin{align*}
m_{\frac{n+1}{n-1}f}\bigl(\widehat{U}^{f}_{r}\bigl)&=\int_{  U^{f}_{r} }\,\int^{\tau(z)}_{0}\,e^{  \frac{-(n+1)f(\gamma_{z}(t))}{n-1}  }\,\theta(t,z)\,dt\, d\vol_{h}\\
                                                                         &=\int_{   U^{f}_{r} }\,\int^{r}_{0}\,\check{\theta}_{f}(s,z)\,ds\,d\vol_{h},\\
m_{\frac{n+1}{n-1}f}\bigl(\widehat{V}^{f}_{r}\bigl)&=\int_{V^{f}_{r} }\,\int^{t_{f,z}(r)}_{0}\,e^{  \frac{-(n+1)f(\gamma_{z}(t))}{n-1}  }\,\theta(t,z)\,dt\, d\vol_{h}\\
                                                                        &=\int_{    V^{f}_{r}}\,\int^{r}_{0}\,\check{\theta}_{f}(s,z)\,ds\,d\vol_{h}.
\end{align*}
Since $\cut \bm$ is a null set,
we conclude the lemma.
\end{proof}

We set
\begin{equation*}
m_{f,\bm}:=e^{-f|_{\bm}} \vol_{h}.
\end{equation*}

We prove the following absolute volume comparison inequality:
\begin{lem}\label{lem:absolute volume comparison}
For $N\in (-\infty,1]$,
assume that
$(M,\bm,f)$ has lower $(\kappa,\lambda,N)$-weighted curvature bounds.
Let $\bm$ be compact.
Then for all $r>0$
we have
\begin{equation}\label{eq:absolute volume comparison}
m_{\frac{n+1}{n-1}f}\left(B^{f}_{r}(\bm)\right) \leq s_{n,\kappa,\lambda}(r)\,m_{f,\bm}(\bm).
\end{equation}
\end{lem}
\begin{proof}
By Lemma \ref{lem:volume element comparison},
$\check{\theta}_{f}(s,z)\leq e^{-f(z)}\,\bar{s}^{n-1}_{\kappa,\lambda}(s)$ for all $s \geq 0$,
where $\bar{s}_{\kappa,\lambda}$ is defined as $(\ref{eq:model volume growth})$.
Integrate both sides over $[0,r]$ with respect to $s$,
and over $\bm$ with respect to $z$.
Lemma \ref{lem:integration formula} implies the lemma.
\end{proof}

One can also prove the following
by replacing the role of Lemmas \ref{lem:volume element comparison} and \ref{lem:integration formula} with Lemmas \ref{lem:finite volume element comparison} and \ref{lem:Basic integration formula} in the proof of Lemma \ref{lem:absolute volume comparison}.
\begin{lem}\label{lem:finite absolute volume comparison}
Let $\kappa$ and $\lambda$ satisfy the monotone-condition.
For $N\in (-\infty,1]$,
assume that
$(M,\bm,f)$ has lower $(\kappa,\lambda,N)$-weighted curvature bounds.
Let $\bm$ be compact.
Suppose additionally that
there is $\delta \in \mathbb{R}$ such that $f\leq (n-1)\delta$ on $M$.
Then for all $r>0$
we have
\begin{equation*}\label{eq:finite absolute volume comparison}
m_{f}(B_{r}(\bm)) \leq s_{n,\kappa\,e^{-4\,\delta},\lambda\,e^{-2\,\delta}}(r)\,m_{f,\bm}(\bm).
\end{equation*}
\end{lem}

\subsection{Relative comparisons}\label{sec:Relative volume comparisons}

We prove Theorem \ref{thm:relative volume comparison}.
\begin{proof}[Proof of Theorem \ref{thm:relative volume comparison}]
For $N \in (-\infty,1]$,
assume that
$(M,\bm,f)$ has lower $(\kappa,\lambda,N)$-weighted curvature bounds.
Let $\bm$ be compact.

By Lemma \ref{lem:volume element comparison},
for all $s_{1},s_{2}\geq 0$ with $s_{1}\leq s_{2}$
\begin{equation*}
\check{\theta}_{f}(s_{2},z)\; \bar{s}_{\kappa,\lambda}^{n-1}(s_{1}) \leq \check{\theta}_{f}(s_{1},z)\; \bar{s}_{\kappa,\lambda}^{n-1}(s_{2}).
\end{equation*}
We integrate the both sides over $[0,r]$ with respect to $s_{1}$,
and over $[r,R]$ with respect to $s_{2}$.
It follows that
\begin{equation*}
\frac{\int^{R}_{r}\check{\theta}_{f}(s_{2},z)\,ds_{2}}{\int^{r}_{0}\check{\theta}_{f}(s_{1},z)\,ds_{1}}\leq \frac{s_{n,\kappa,\lambda}(R)-s_{n,\kappa,\lambda}(r)}{s_{n,\kappa,\lambda}(r)}.
\end{equation*}
Lemma \ref{lem:integration formula} implies
\begin{equation*}
          \frac{m_{\frac{n+1}{n-1}f}( B^{f}_{R}(\bm))}{m_{\frac{n+1}{n-1}f}(B^{f}_{r}(\bm))}\leq 1+\frac{s_{n,\kappa,\lambda}(R)-s_{n,\kappa,\lambda}(r)}{s_{n,\kappa,\lambda}(r)}
    =      \frac{s_{n,\kappa,\lambda}(R)}{s_{n,\kappa,\lambda}(r)},
\end{equation*}
and hence (\ref{eq:relative volume comparison}).
We complete the proof of Theorem \ref{thm:relative volume comparison}.
\end{proof}
\begin{rem}\label{rem:half rigidity}
Suppose that
there exists $R\in(0,\bar{C}_{\kappa,\lambda}]\setminus \{\infty\}$ such that
for every $r\in (0,R]$
the equality in $(\ref{eq:relative volume comparison})$ holds.
Then we see $\tau_{f} \geq R$ on $\bm$ (cf. Lemma 4.3 in \cite{Sa3}).
\end{rem}

We can also prove the following volume comparison by using Lemmas \ref{lem:finite volume element comparison} and \ref{lem:Basic integration formula}
instead of Lemmas \ref{lem:volume element comparison} and \ref{lem:integration formula} in the proof of Theorem \ref{thm:relative volume comparison}.
\begin{thm}\label{thm:finite relative volume comparison}
Let $\kappa$ and $\lambda$ satisfy the monotone-condition.
For $N\in (-\infty,1]$,
assume that
$(M,\bm,f)$ has lower $(\kappa,\lambda,N)$-weighted curvature bounds.
Let $\bm$ be compact.
Suppose additionally that
there is $\delta \in \mathbb{R}$ such that $f\leq (n-1)\delta$ on $M$.
Then for all $r,R>0$ with $r\leq R$
\begin{equation}\label{eq:finite relative volume comparison}
\frac{m_{f}(B_{R}(\bm))}{m_{f}(B_{r}(\bm))} \leq \frac{s_{n,\kappa\,e^{-4\delta},\lambda\,e^{-2\delta}}(R)}{s_{n,\kappa\,e^{-4\delta},\lambda\,e^{-2\delta}}(r)}.
\end{equation}
\end{thm}
\begin{rem}\label{rem:finite half rigidity}
Assume that
there exists $R\in(0,\bar{C}_{\kappa\,e^{-4\delta},\lambda\,e^{-2\delta}}]\setminus \{\infty\}$ such that
for every $r\in (0,R]$
the equality in $(\ref{eq:finite relative volume comparison})$ holds.
Then one can verify that $\tau \geq R$ on $\bm$ (cf. Lemma 4.3 in \cite{Sa3}).
\end{rem}

\subsection{Volume growth rigidity}\label{sec:Volume growth rigidity}
Now,
let us prove the following:
\begin{thm}\label{thm:volume growth rigidity}
Suppose that
$\kappa$ and $\lambda$ do not satisfy the ball-condition.
For $N \in (-\infty,1]$,
assume that
$(M,\bm,f)$ has lower $(\kappa,\lambda,N)$-weighted curvature bounds.
Let $\bm$ be compact.
If we have
\begin{equation}\label{eq:assumption of volume growth rigidity}
\liminf_{r\to \infty}\frac{m_{\frac{n+1}{n-1}f}(B^{f}_{r}(\bm))}{s_{n,\kappa,\lambda}(r)}\geq m_{f,\bm}(\bm),
\end{equation}
then $M$ is isometric to $[0,\infty)\times_{F_{\kappa,\lambda}} \bm$;
moreover,
if $N \in (-\infty,1)$,
then for every $z\in \bm$
the function $f\circ \gamma_{z}$ is constant.
\end{thm}
\begin{proof}
Lemma \ref{lem:absolute volume comparison} and Theorem \ref{thm:relative volume comparison} imply that
for every $R>0$,
and for every $r\in (0,R]$
the equality in (\ref{eq:relative volume comparison}) holds.
Then $\tau_{f}=\infty$ on $\bm$ (see Remark \ref{rem:half rigidity}).
In particular,
we have $\tau=\infty$ on $\bm$.

Fix $z\in\bm$.
For all $s\geq 0$
we see $\check{\theta}_{f}(s,z)= e^{-f(z)}\,s^{n-1}_{\kappa,\lambda}(s)$.
Choose an orthonormal basis $\{e_{z,i}\}_{i=1}^{n-1}$ of $T_{z}\bm$,
and let $\{Y_{z,i}\}^{n-1}_{i=1}$ be the $\bm$-Jacobi fields along $\gamma_{z}$ with $Y_{z,i}(0)=e_{z,i},\,Y_{z,i}'(0)=-A_{u_{z}}e_{z,i}$.
For all $i$
we have $Y_{z,i}=F_{\kappa,\lambda,z}\,E_{z,i}$,
where $\{E_{z,i}\}^{n-1}_{i=1}$ are the parallel vector fields $E_{z,i}(0)=e_{z,i}$.
Moreover,
if $N \in (-\infty,1)$,
then $f \circ \gamma_{z}$ is constant (see Remark \ref{rem:equality in volume element comparison}).
By the rigidity of Jacobi fields,
a map $\Phi:[0,\infty)\times \bm\to M$ defined by $\Phi(t,z)=\gamma_{z}(t)$ gives a desired isometry.
\end{proof}
\begin{rem}
If $\kappa$ and $\lambda$ satisfy the ball-condition,
then the author does not know whether a similar result to Theorem \ref{thm:volume growth rigidity} holds.
In this case,
under the same setting as in Theorem \ref{thm:volume growth rigidity},
Lemma \ref{lem:Cut point comparisons} implies $\tau_{f}=\const$ on $\bm$ (see Remark \ref{rem:half rigidity}).
Since $\tau(z)$ can be either finite or infinite for each $z\in \bm$,
it seems to be difficult to conclude any rigidity results.
\end{rem}

Next,
we prove the following volume growth rigidity:
\begin{thm}\label{thm:finite volume growth rigidity}
Let $\kappa$ and $\lambda$ satisfy the monotone-condition.
For $N\in (-\infty,1]$,
assume that
$(M,\bm,f)$ has lower $(\kappa,\lambda,N)$-weighted curvature bounds.
Let $\bm$ be compact.
Suppose additionally that
there exists $\delta \in \mathbb{R}$ such that $f\leq (n-1)\delta$ on $M$.
If
\begin{equation*}\label{eq:assumption of finite volume growth rigidity}
\liminf_{r\to \infty}\frac{m_{f}(B_{r}(\bm))}{s_{n,\kappa\,e^{-4\,\delta},\lambda\,e^{-2\,\delta}}(r)}\geq m_{f,\bm}(\bm),
\end{equation*}
then the following hold:
\begin{enumerate}
\item if $\kappa$ and $\lambda$ satisfy the convex-ball-condition,
         then $M$ is isometric to $B^{n}_{\kappa e^{-4\delta},\lambda e^{-2\delta}}$,
         and $f=(n-1)\delta$ on $M$;
\item if $\kappa \leq 0$ and $\lambda=\sqrt{\vert \kappa \vert}$,
         then $M$ is isometric to $[0,\infty)\times_{F_{\kappa,\lambda}} \bm$;
         moreover,
         the following hold:
         \begin{enumerate}
         \item if $\kappa=0$ and $N \in (-\infty,1)$,
                  then for every $z\in \bm$
                  the function $f\circ \gamma_{z}$ is constant on $[0,\infty)$;
         \item if $\kappa<0$,
                  then $f=(n-1)\delta$ on $M$.
         \end{enumerate}
\end{enumerate}
\end{thm}
\begin{proof}
In view of Lemma \ref{lem:finite absolute volume comparison} and Theorem \ref{thm:finite relative volume comparison},
we see the following:
If $\kappa$ and $\lambda$ satisfy the convex-ball-condition,
then for $R=C_{\kappa\,e^{-4\delta},\lambda\,e^{-2\delta}}$,
and for every $r\in (0,R]$ the equality in (\ref{eq:finite relative volume comparison}) holds;
in particular,
$\tau=C_{\kappa\,e^{-4\delta},\lambda\,e^{-2\delta}}$ on $\bm$ (see Remark \ref{rem:finite half rigidity}).
If $\kappa \leq 0$ and $\lambda=\sqrt{\vert \kappa \vert}$,
then for every $R>0$,
and for every $r\in (0,R]$
the equality in (\ref{eq:finite relative volume comparison}) holds;
in particular,
$\tau=\infty$ (see Remark \ref{rem:finite half rigidity}).
Hence,
$\tau=\bar{C}_{\kappa\,e^{-4\delta},\lambda\,e^{-2\delta}}$ on $\bm$.

If $\kappa$ and $\lambda$ satisfy the convex-ball-condition,
then due to Theorem \ref{thm:finite inscribed radius rigidity},
$M$ is isometric to $B^{n}_{\kappa e^{-4\delta},\lambda e^{-2\delta}}$,
and $f=(n-1)\delta$.

If $\kappa \leq 0$ and $\lambda=\sqrt{\vert \kappa \vert}$,
then $\cut \bm=\emptyset$.
Theorem \ref{thm:splitting theorem} tells us that
$M$ is isometric to $[0,\infty)\times_{F_{\kappa,\lambda}} \bm$;
moreover,
if $N\in (-\infty,1)$,
then for every $z\in \bm$
the function $f\circ \gamma_{z}$ is constant.
In the case of $\kappa<0$,
for all $t\geq 0,\,z\in \bm$
we see $\theta_{f}(t,z)=e^{-f(z)}s^{n-1}_{\kappa\,e^{-4\delta},\lambda\,e^{-2\delta}}(t)$;
in particular,
$f=(n-1)\delta$ (see Remark \ref{rem:equality in finite volume element comparison}).
We complete the proof.
\end{proof}

\section{Eigenvalues}\label{sec:Eigenvalues}

\subsection{Lower bounds}
We recall the following inequality of Picone type (see Theorem 1.1 in \cite{AH}, and Lemma 7.1 in \cite{Sa2}):
\begin{lem}[\cite{AH}, \cite{Sa2}]\label{lem:Picone identity}
Let $p\in (1,\infty)$.
Let $\phi>0$ and $\psi \geq 0$ be two $C^{1}$-functions on a domain $U \subset M$.
Then we have
\begin{equation}\label{eq:Picone identity}
\Vert \nabla \psi \Vert^{p}\geq \Vert \nabla \phi \Vert^{p-2} g\left(\nabla \left(  \psi^{p}\,  \phi^{1-p} \right),\nabla \phi \right).
\end{equation}
If the equality in $(\ref{eq:Picone identity})$ holds on $U$,
then $\psi=c\, \phi$ for some $c\neq 0$ on $U$.
\end{lem}

We prove the inequality $(\ref{eq:eigenvalue rigidity})$ in Theorem \ref{thm:eigenvalue rigidity}.
\begin{lem}\label{lem:inequality in eigenvalue rigidity}
Let $p\in (1,\infty)$.
For $N\in (-\infty,1]$,
assume that
$(M,\bm,f)$ has lower $(\kappa,\lambda,N)$-weighted curvature bounds.
Let $M$ be compact,
and let $f$ be $\bm$-radial.
Suppose additionally that
there is $\delta \in \mathbb{R}$ such that $f\leq (n-1)\delta$ on $M$.
For $D\in (0,\bar{C}_{\kappa,\lambda}]\setminus \{\infty\}$,
suppose $\IR_{f} M \leq D$,
where $\IR_{f} M$ is defined as $(\ref{eq:twisted inscribed radius})$.
Then we have $(\ref{eq:eigenvalue rigidity})$.
\end{lem}
\begin{proof}
First,
we notice $\nu_{p,\kappa\,e^{-4\delta},\lambda\,e^{-2\delta},D\,e^{2\delta}}=e^{-2p\delta}\nu_{p,\kappa,\lambda,D}$.
Let $\hat{\varphi}$ be a non-zero function satisfying (\ref{eq:model space eigenvalue problem}) for $\nu=\nu_{p,\kappa,\lambda,D}$.
We may assume $\hat{\varphi}|_{(0,D]}>0$.
The equation (\ref{eq:model space eigenvalue problem}) can be written in the form
\begin{align*}
\left(\vert \varphi'(s)\vert^{p-2} \varphi'(s) s^{n-1}_{\kappa,\lambda}(s) \right)'+\nu\, \vert \varphi(s)\vert^{p-2}\varphi(s) s^{n-1}_{\kappa,\lambda}(s)=0,\\
                                                              \varphi(0)=0, \quad \varphi'(D)=0.                           \notag
\end{align*}
Hence
we see $\hat{\varphi}'|_{[0,D)}>0$.
Put $\Phi:=\hat{\varphi} \circ \rho_{\bm,f}$,
and fix a non-negative, non-zero function $\psi \in C^{\infty}_{0}(M)$.
From Lemma \ref{lem:Picone identity}
we deduce
\begin{equation}\label{eq:radial Picone identity}
\Vert \nabla \psi \Vert^{p}\geq \Vert \nabla \Phi \Vert^{p-2} g\left(\nabla \left(  \psi^{p}\,  \Phi^{1-p} \right),\nabla \Phi \right)
\end{equation}
on $\inte M \setminus \cut \bm$.
Here
we put $\hat{f}:=(n+1)\,(n-1)^{-1}\,f$ and $\check{f}:=(n+1-2p)\,(n-1)^{-1}\,f$.
By $f \leq (n-1)\delta$ and (\ref{eq:radial Picone identity}),
\begin{align}\label{eq:bounded and radial Picone identity}
e^{2p\delta}\,\int_{M}\, \Vert \nabla \psi \Vert^{p}\, d\,m_{\hat{f}} &\geq \int_{M}\,e^{\frac{2pf}{n-1}}\, \Vert \nabla \psi \Vert^{p}\, d\,m_{\hat{f}}=\int_{M}\, \Vert \nabla \psi \Vert^{p}\, d\,m_{\check{f}}\\ \notag
                                                                                                    &\geq \int_{M}\,\Vert \nabla \Phi \Vert^{p-2} g\left(\nabla \left(  \psi^{p}\,  \Phi^{1-p} \right),\nabla \Phi \right) \,d\,m_{\check{f}}.
\end{align}
Further,
(\ref{eq:bounded and radial Picone identity}) and (\ref{eq:global radial p-Laplacian comparison}) tell us that
$e^{2p\delta}\,\int_{M}\, \Vert \nabla \psi \Vert^{p}\, d\,m_{\hat{f}}$ is at least
\begin{equation*}
-\int_{M}  \psi^{p}\,  \Phi^{1-p} \,  \left\{ \left( \bigl( (\hat{\varphi}' )^{p-1} \bigl)' -H_{\kappa,\lambda}\, \left(\hat{\varphi}' \right)^{p-1}\right)\circ \rho_{\bm,f}\right\}  dm_{ \hat{f} }
\end{equation*}
that is equal to $\nu_{p,\kappa,\lambda,D}\,\int_{M}\, \psi^{p} \, d\,m_{\hat{f}}$ by the definition of $H_{\kappa,\lambda}$ (see (\ref{eq:model space eigenvalue problem}) and (\ref{eq:model mean curvature})).
Therefore,
$R_{\hat{f},p}(\psi)\geq e^{-2p\delta}\nu_{p,\kappa,\lambda,D}$.
This implies (\ref{eq:eigenvalue rigidity}).
\end{proof}

\begin{rem}\label{rem:the equality case in eigenvalue rigidity}
From the argument in the proof of Lemma \ref{lem:inequality in eigenvalue rigidity},
one can also verify the following:
Under the same setting as in Lemma \ref{lem:inequality in eigenvalue rigidity},
we have $R_{\hat{f},p}(\psi)\geq \nu_{p,\kappa\,e^{-4\delta},\lambda\,e^{-2\delta},D\,e^{2\delta}}$ for every non-negative, non-zero $\psi \in C^{1}(M)$ with $\psi|_{\bm}=0$ (cf. Remark \ref{rem:c1-equality case in global radial p-Laplacian comparison}).
Moreover,
if the equality holds for some $\psi$,
then the equalities in (\ref{eq:bounded and radial Picone identity}) hold,
and hence $f=(n-1)\delta$ on the set where $\nabla \psi \neq 0$,
and $\psi=c\, \Phi$ for some $c\neq 0$ (see Lemma \ref{lem:Picone identity});
in particular,
we can conclude $\nabla \psi \neq 0$ on $M \setminus \cut \bm$,
and $f = (n-1)\delta$ on $M$.
\end{rem}

We next prove the following comparison inequality:
\begin{lem}\label{lem:inequality in finite eigenvalue rigidity}
Let $p\in (1,\infty)$.
Let $\kappa$ and $\lambda$ satisfy the convex-ball-condition.
For $N\in (-\infty,1]$,
let us assume that
$(M,\bm,f)$ has lower $(\kappa,\lambda,N)$-weighted curvature bounds.
Let $M$ be compact.
Suppose additionally that
there exists $\delta \in \mathbb{R}$ such that $f\leq (n-1)\delta$ on $M$.
Then we have $\nu_{f,p}(M)\geq \nu_{0,p}(B^{n}_{\kappa e^{-4\delta},\lambda e^{-2\delta}})$.
\end{lem}
\begin{proof}
We first note that $\nu_{0,p}(B^{n}_{\kappa\,e^{-4\,\delta},\lambda\,e^{-2\,\delta}})=e^{-2p\delta}\,\nu_{p,\kappa,\lambda,\const}$.
Let $\hat{\varphi}:[0,\const]\to \mathbb{R}$ be a non-zero function satisfying (\ref{eq:model space eigenvalue problem}) for $\nu=\nu_{p,\kappa,\lambda,\const}$,
and let $\hat{\varphi}|_{(0,\const]}>0$.
We see $\hat{\varphi}'|_{[0,\const)}>0$.
Define functions $\rho_{\bm,\delta}:=e^{-2\delta}\,\rho_{\bm}$ and $\Phi:=\hat{\varphi} \circ \rho_{\bm,\delta}$.
We fix a non-negative,
non-zero function $\psi \in C^{\infty}_{0}(M)$.
Then Lemma \ref{lem:Picone identity} leads us to
\begin{equation}\label{eq:finite Picone identity}
\Vert \nabla \psi \Vert^{p}\geq \Vert \nabla \Phi \Vert^{p-2} g\left(\nabla \left(  \psi^{p}\,  \Phi^{1-p} \right),\nabla \Phi \right)
\end{equation}
on $\inte M \setminus \cut \bm$.
Notice that
$\kappa$ and $\lambda$ satisfy the model-condition.
From (\ref{eq:finite Picone identity}) and (\ref{eq:global finite p-Laplacian comparison}),
it follows that
\begin{align*}
&\quad \, \,e^{2p\delta}\int_{M}\, \Vert \nabla \psi \Vert^{p}\, d\,m_{f}   \geq    e^{2p\delta}\int_{M}\,\Vert \nabla \Phi \Vert^{p-2} g\left(\nabla \left(  \psi^{p}\,  \Phi^{1-p} \right),\nabla \Phi \right) \,d\,m_{f}\\
&\geq -\int_{M}\, \psi^{p}\,  \Phi^{1-p} \,\left\{   \left( \bigl( (\hat{\varphi}' )^{p-1} \bigl)'-H_{\kappa,\lambda}\,\left(\hat{\varphi}' \right)^{p-1}\right)\circ \rho_{\bm,\delta}\right\} \,d\,m_{f}.
\end{align*}
The right hand side is equal to $\nu_{p,\kappa,\lambda,\const}\,\int_{M}\, \psi^{p} \, d\,m_{f}$.
Therefore,
we obtain $R_{f,p}(\psi)\geq e^{-2p\delta}\,\nu_{p,\kappa,\lambda,\const}$.
This proves the lemma.
\end{proof}

\begin{rem}\label{rem:the equality case in finite eigenvalue rigidity}
From the argument in the proof of Lemma \ref{lem:inequality in finite eigenvalue rigidity},
we can also conclude the following:
Under the same setting as in Lemma \ref{lem:inequality in finite eigenvalue rigidity},
we have $R_{f,p}(\psi)\geq \nu_{0,p}(B^{n}_{\kappa\,e^{-4\,\delta},\lambda\,e^{-2\,\delta}})$ for every non-negative, non-zero $\psi \in C^{1}(M)$ with $\psi|_{\bm}=0$ (cf. Remark \ref{rem:c1-equality case in global finite p-Laplacian comparison}).
Moreover,
if the equality holds for some $\psi$,
then the equalities in (\ref{eq:finite Picone identity}) and (\ref{eq:global finite p-Laplacian comparison}) hold (see Lemma \ref{lem:Picone identity} and Remark \ref{rem:c1-equality case in global finite p-Laplacian comparison});
in particular,
$\psi=c\, \Phi$ for some $c\neq 0$ on $M$.
\end{rem}

\subsection{Equality cases}
We recall the following fact for eigenfunctions of the weighted $p$-Laplacian (see e.g., \cite{T}):
\begin{prop}[\cite{T}]\label{prop:eigenfunction}
Let $p\in (1,\infty)$.
Let $\phi:M\to \mathbb{R}$ be a smooth function.
Let $M$ be compact.
Then there exists a non-negative,
non-zero function $\psi \in W^{1,p}_{0}(M,m_{\phi})$ such that $R_{\phi,p}(\psi)=\nu_{\phi,p}(M)$.
Moreover,
$\psi \in C^{1,\alpha}(M)$ for some $\alpha \in (0,1)$.
\end{prop}

By using Proposition \ref{prop:eigenfunction},
we prove Theorem \ref{thm:eigenvalue rigidity}.
\begin{proof}[Proof of Theorem \ref{thm:eigenvalue rigidity}]
Let $p\in (1,\infty)$.
For $N\in (-\infty,1]$,
assume that
$(M,\bm,f)$ has lower $(\kappa,\lambda,N)$-weighted curvature bounds.
Let $M$ be compact,
and let $f$ be $\bm$-radial.
Suppose additionally that
there is $\delta \in \mathbb{R}$ such that $f\leq (n-1)\delta$ on $M$.
For $D\in (0,\bar{C}_{\kappa,\lambda}]\setminus \{\infty\}$,
suppose $\IR_{f} M \leq D$.
Lemma \ref{lem:inequality in eigenvalue rigidity} yields (\ref{eq:eigenvalue rigidity}).

Let us assume that
the equality in (\ref{eq:eigenvalue rigidity}) holds.
By applying Proposition \ref{prop:eigenfunction} to $(n+1)(n-1)^{-1}f$,
there exists a non-negative, non-zero $\psi \in W^{1,p}_{0}(M,m_{\frac{n+1}{n-1}f}) \cap C^{1,\alpha}(M)$ with $R_{\frac{n+1}{n-1}f,p}(\psi)=\nu_{p,\kappa\,e^{-4\delta},\lambda \,e^{-2\delta},D\,e^{2\delta}}$.
Then $f = (n-1)\delta$ on $M$ (see Remark \ref{rem:the equality case in eigenvalue rigidity}).
Theorem \ref{thm:eigenvalue rigidity} has been already known when $f$ is constant (see Theorem 1.6 in \cite{Sa2}).
Thus,
we complete the proof of Theorem \ref{thm:eigenvalue rigidity}.
\end{proof}

\begin{rem}
Kasue \cite{K3} has obtained an explicit lower bound for $\mu_{2,\kappa,\lambda,D}$ (see Lemma 1.3 in \cite{K3}).
Due to the estimate,
under the same setting as in Theorem \ref{thm:eigenvalue rigidity} with $p=2$,
we have an explicit bound for $\nu_{\frac{n+1}{n-1}f,2}(M)$.
\end{rem}

We also formulate the following eigenvalue rigidity theorem:
\begin{thm}\label{thm:finite eigenvalue rigidity}
Let $p\in (1,\infty)$.
Let $\kappa$ and $\lambda$ satisfy the convex-ball-condition.
For $N\in (-\infty,1]$,
let us assume that
$(M,\bm,f)$ has lower $(\kappa,\lambda,N)$-weighted curvature bounds.
Let $M$ be compact.
Suppose additionally that
there is $\delta \in \mathbb{R}$ such that $f\leq (n-1)\delta$ on $M$.
Then
\begin{equation}\label{eq:finite eigenvalue rigidity}
\nu_{f,p}(M)\geq \nu_{0,p}(B^{n}_{\kappa e^{-4\delta},\lambda e^{-2\delta}}).
\end{equation}
If the equality in $(\ref{eq:finite eigenvalue rigidity})$ holds,
then $M$ is isometric to $B^{n}_{\kappa e^{-4\delta},\lambda e^{-2\delta}}$,
and $f=(n-1)\delta$ on $M$.
\end{thm}
\begin{proof}
By Lemma \ref{lem:inequality in finite eigenvalue rigidity},
we have (\ref{eq:finite eigenvalue rigidity}).
Assume that
the equality holds.
Applying Proposition \ref{prop:eigenfunction} to $f$,
we have a non-negative, 
non-zero $\psi \in W^{1,p}_{0}(M,m_{f})\cap C^{1,\alpha}(M)$ with $R_{f,p}(\psi)=\nu_{0,p}(B^{n}_{\kappa\,e^{-4\,\delta},\lambda\,e^{-2\,\delta}})$.
Let $\hat{\varphi}$ be a non-zero function satisfying (\ref{eq:model space eigenvalue problem}) for $\nu=\nu_{p,\kappa,\lambda,\const}$,
and let $\hat{\varphi}|_{(0,\const]}>0$.
Define $\rho_{\bm,\delta}:=e^{-2\,\delta}\,\rho_{\bm}$ and $\Phi:=\hat{\varphi} \circ \rho_{\bm,\delta}$ (cf. Lemma \ref{lem:finite Inscribed radius comparison}).
Then $\Phi=c\,\psi$ for some $c\neq 0$;
in particular,
$\supp \psi =M$ and $\Phi \in C^{1,\alpha}(M)$.
The equality in (\ref{eq:global finite p-Laplacian comparison}) also holds (see Remark \ref{rem:the equality case in finite eigenvalue rigidity}).

Since $\supp \psi =M$,
the equality in (\ref{eq:finite p-Laplacian comparison}) holds on $\inte M \setminus \cut \bm$ (see Remark \ref{rem:c1-equality case in global finite p-Laplacian comparison}).
Fix $z\in \bm$.
Choose an orthonormal basis $\{e_{z,i}\}_{i=1}^{n-1}$ of $T_{z}\bm$.
Let $\{Y_{z,i}\}^{n-1}_{i=1}$ be the $\bm$-Jacobi fields along $\gamma_{z}$ with $Y_{z,i}(0)=e_{z,i},\,Y_{z,i}'(0)=-A_{u_{z}}e_{z,i}$.
For all $i$
we see $Y_{z,i}=F_{\kappa,\lambda,z}\, E_{z,i}$ on $[0,\tau(z)]$,
where $\{E_{z,i}\}^{n-1}_{i=1}$ are the parallel vector fields with $E_{z,i}(0)=e_{z,i}$.
Moreover,
$f\circ \gamma_{z}=(n-1)\delta$ on $[0,\tau(z)]$.

By Theorem \ref{thm:finite inscribed radius rigidity},
it suffices to show that $\IR M=C_{\kappa\,e^{-4\delta},\lambda\,e^{-2\delta}}$.
Let us suppose $\IR M<C_{\kappa\,e^{-4\delta},\lambda\,e^{-2\delta}}$.
Take $x_{0}\in M$ with $\rho_{\bm}(x_{0})=\IR M$.
Note that
$x_{0} \in \cut \bm$.
By $\rho_{\bm}(x_{0})<C_{\kappa\,e^{-4\delta},\lambda\,e^{-2\delta}}$,
and by the rigidity of the Jacobi fields,
$x_{0}$ is not the first conjugate point along $\gamma_{z_{0}}$,
where $z_{0}$ is a foot point of $x_{0}$.
Hence
$\rho_{\bm,\delta}$ is not differentiable at $x_{0}$.
From $\Phi \in C^{1,\alpha}(M)$
we deduce $\hat{\varphi}'(\rho_{\bm,\delta}(x_{0}))=0$.
This contradicts $\hat{\varphi}'|_{[0,\const)}>0$.
Thus,
we complete the proof of Theorem \ref{thm:finite eigenvalue rigidity}.
\end{proof}

\subsection{Spectrum rigidity}\label{sec:Spectrum rigidity}
Let $\Omega$ be a relatively compact domain in $M$ such that
its boundary is a smooth hypersurface in $M$ with $\partial \Omega \cap \bm=\emptyset$.
For the canonical measure $\vol_{\partial \Omega}$ on $\partial \Omega$,
put $m_{f,\partial \Omega}:=e^{-f|_{\partial \Omega}}\, \vol_{\partial \Omega}$.

Let us prove the following area estimate:
\begin{lem}\label{lem:Kasue volume estimate}
Let $\kappa$ and $\lambda$ satisfy the monotone-condition.
For $N\in (-\infty,1]$,
assume that
$(M,\bm,f)$ has lower $(\kappa,\lambda,N)$-weighted curvature bounds.
Suppose additionally that
there exists $\delta \in \mathbb{R}$ such that $f\leq (n-1)\delta$ on $M$.
Define $\rho_{\bm,\delta}:=e^{-2\delta}\,\rho_{\bm}$.
Let $\Omega$ be a relatively compact domain in $M$ such that
$\partial \Omega$ is a smooth hypersurface in $M$ satisfying $\partial \Omega \cap \bm=\emptyset$.
Set
\begin{equation*}\label{eq:diameter of Omega}
D_{\delta,1}(\Omega):=\inf_{x\in \Omega}\, \rho_{\bm,\delta}(x),\quad D_{\delta,2}(\Omega):=\sup_{x\in \Omega} \,\rho_{\bm,\delta}(x).
\end{equation*}
Then we have
\begin{equation*}\label{eq:Kasue volume estimate}
m_{f}( \Omega) \leq e^{2\delta}\,  \sup_{s\in (D_{\delta,1}(\Omega),D_{\delta,2}(\Omega))}\,  \frac{\int^{ D_{\delta,2}(\Omega)}_{s}\,  s^{n-1}_{\kappa,\lambda}(a)\, da}{s^{n-1}_{\kappa,\lambda}(s)} \, m_{f,\partial \Omega}\, (\partial \Omega).
\end{equation*}
\end{lem}
\begin{proof}
Define a function $\hat{\varphi}:[D_{\delta,1}(\Omega),D_{\delta,2}(\Omega)]\to \mathbb{R}$ by
\begin{align*}
\hat{\varphi}(s):=\int^{s}_{D_{\delta,1}(\Omega) }\,   \frac{\int^{ D_{\delta,2}(\Omega)}_{a}\,  s^{n-1}_{\kappa,\lambda}(b)\, db}{s^{n-1}_{\kappa,\lambda}(a)}   \,da.
\end{align*}
Put $\Phi:=\hat{\varphi} \circ \rho_{\bm,\delta}$.
By Lemma \ref{lem:finite p-Laplacian comparison},
on $\Omega \setminus \cut\bm$
we have
\begin{equation}\label{eq:Kasue volume estimate Laplacian comparison}
\Delta_{f}\,\Phi \geq -e^{-4\delta}\,  \left(  \hat{\varphi}''-H_{\kappa,\lambda}\,\hat{\varphi}' \right)\circ \rho_{\bm,\delta}=e^{-4\delta}.
\end{equation}

By Lemma \ref{lem:avoiding the cut locus},
there exists a sequence $\{\Omega_{i}\}$ of compact subsets of the closure $\bar{\Omega}$ such that
for every $i$,
the boundary $\partial \Omega_{i}$ is a smooth hypersurface in $M$ except for a null set in $(\partial \Omega,m_{f,\partial \Omega})$,
and satisfying the following:
(1) for all $i_{1},i_{2}$ with $i_{1}<i_{2}$,
we have $\Omega_{i_{1}}\subset \Omega_{i_{2}}$;
(2) $\bar{\Omega} \setminus \cut \bm=\bigcup_{i}\,\Omega_{i}$:
(3) for every $i$,
     and for almost every point $x \in \partial \Omega_{i}\cap \partial \Omega$ in $(\partial \Omega,m_{f,\partial \Omega})$,
     there exists a unique unit outer normal vector for $\Omega_{i}$ at $x$
     that coincides with the unit outer normal vector $u_{\partial \Omega}$ on $\partial \Omega$ for $\Omega$;
(4) for every $i$,
     on $\partial \Omega_{i}\setminus \partial \Omega$,
     there exists a unique unit outer normal vector field $u_{i}$ for $\Omega_{i}$ such that $g(u_{i},\nabla \rho_{\bm})\geq 0$.

For the canonical measure $\vol_{i}$ on $\partial \Omega_{i}\setminus \partial \Omega$,
put $m_{f,i}:=e^{-f|_{\partial \Omega_{i}\setminus \partial \Omega}}\,\vol_{i}$.
By integrating the both sides of (\ref{eq:Kasue volume estimate Laplacian comparison}) on $\Omega_{i}$,
and by integration by parts,
\begin{align*}
&\,  e^{-4\delta}\,m_{f}\left(\Omega_{i}\right) \leq \int_{\Omega_{i}} \, \Delta_{f}\,\Phi \,d\,m_{f}\\
= &-\int_{\partial \Omega_{i}\setminus \partial \Omega} g(u_{i},\nabla \Phi)  \,d\,m_{f,i}-\int_{\partial \Omega_{i} \cap \partial \Omega} g(u_{\partial \Omega},\nabla \Phi)  \,d\,m_{f,\partial \Omega}.
\end{align*}
The Cauchy-Schwarz inequality and $g(u_{i},\nabla \Phi)\geq 0$ tell us that
\begin{align*}
e^{-4\delta}\,m_{f}\left(\Omega_{i}\right) &\leq -\int_{\partial \Omega_{i}\cap \partial \Omega} g(u_{\partial \Omega},\nabla \Phi)  \,d\,m_{f,\partial \Omega}\\
&\leq \int_{\partial \Omega_{i}\cap \partial \Omega} \left( \hat{\varphi}'\circ \rho_{\bm,\delta}\right) \vert g(u_{\partial \Omega},\nabla \rho_{\bm,\delta}) \vert  \,d\,m_{f,\partial \Omega}\\
&\leq e^{-2\delta}\,\sup_{s\in (D_{\delta,1}(\Omega),D_{\delta,2}(\Omega))}\,\hat{\varphi}'(s)\,m_{f,\partial \Omega}\, (\partial \Omega).                                                   
\end{align*}
By letting $i \to \infty$,
we complete the proof.
\end{proof}
Kasue \cite{K4} has proved Lemma \ref{lem:Kasue volume estimate} when $f=0$ and $\delta=0$.

For $\alpha>0$,
the \textit{$(f,\alpha)$-Dirichlet isoperimetric constant} is defined as
\begin{equation*}
DI_{\alpha}(M,m_{f}):=\inf_{\Omega}\,  \frac{m_{f,\partial \Omega}(\partial \Omega)}{\left( m_{f}(\Omega) \right)^{1/\alpha}},
\end{equation*}
where the infimum is taken over all relatively compact domains $\Omega$ in $M$ such that
$\partial \Omega$ are smooth hypersurfaces in $M$ satisfying $\partial \Omega \cap \partial M=\emptyset$.
The \textit{$(f,\alpha)$-Dirichlet Sobolev constant} is defined as
\begin{equation*}
DS_{\alpha}(M,m_{f}):=\inf_{\phi \in W^{1,1}_{0}(M,m_{f})\setminus \{0\}}\, \frac{\int_{M}\,  \Vert \nabla \phi \Vert \,d\,m_{f}}{\left(\int_{M}\,  \vert \phi \vert^{\alpha} \,d\,m_{f} \right)^{1/\alpha}}.
\end{equation*}

Let us recall the following relation between the constants:
\begin{prop}[\cite{FF}]\label{thm:isoperimetric and Sobolev}
For all $\alpha>0$,
$DI_{\alpha}(M,m_{f})=DS_{\alpha}(M,m_{f})$.
\end{prop}

For $D\in (0,\bar{C}_{\kappa,\lambda}]$,
we put
\begin{equation}\label{eq:spectrum constant}
C(\kappa,\lambda,D):=\sup_{s\in [0,D)}\, \frac{\int^{D}_{s}\,  s^{n-1}_{\kappa,\lambda}(a)\, da}{s^{n-1}_{\kappa,\lambda}(s)}.
\end{equation}
Notice that 
$C(\kappa,\lambda,\infty)$ is finite if and only if $\kappa<0$ and $\lambda=\sqrt{\vert \kappa \vert}$;
in this case,
we have
$C(\kappa,\lambda,D)=\left((n-1)\lambda \right)^{-1}\,\left(1-e^{-(n-1)\lambda\, D} \right)$.

From Lemma \ref{lem:Kasue volume estimate}
we derive the following:
\begin{lem}\label{lem:p-Laplacian1}
Let $p\in (1,\infty)$.
Let $\kappa$ and $\lambda$ satisfy the monotone-condition.
For $N\in (-\infty,1]$,
assume that
the triple $(M,\bm,f)$ has lower $(\kappa,\lambda,N)$-weighted curvature bounds.
Suppose additionally that
there exists $\delta \in \mathbb{R}$ such that $f\leq (n-1)\delta$ on $M$.
For $D\in(0,\bar{C}_{\kappa,\lambda}]$,
suppose $\IR M \leq e^{2\,\delta}\,D$.
Then we have
\begin{equation*}
\nu_{f,p}(M)\geq (\,p\,e^{2\delta}\,C(\kappa,\lambda,D)\,)^{-p}.
\end{equation*}
\end{lem}
\begin{proof}
Let $\Omega$ be a relatively compact domain in $M$ such that
$\partial \Omega$ is a smooth hypersurface in $M$ with $\partial \Omega \cap \bm=\emptyset$.
Set $C_{\delta}:=e^{2\delta}\,C(\kappa,\lambda,D)$.
Lemma \ref{lem:Kasue volume estimate} implies $m_{f}(\Omega) \leq C_{\delta}\,m_{f,\partial \Omega}( \partial \Omega)$.
By Proposition \ref{thm:isoperimetric and Sobolev},
we obtain $DS_{1}(M,m_{f})\geq C^{-1}_{\delta}$.
For all $\phi \in W^{1,1}_{0}(M,m_{f})$
we have
\begin{equation}\label{eq:Poincare1}
\int_{M}\, \vert \phi \vert \,d\,m_{f} \leq C_{\delta} \int_{M}\,\Vert \nabla \phi \Vert \,d\,m_{f}.
\end{equation}

Let $\psi$ be a non-zero function in $W^{1,p}_{0}(M,m_{f})$.
Put $q:=p\,(1-p)^{-1}$.
In (\ref{eq:Poincare1}),
by replacing $\phi$ with $\vert \psi \vert^{p}$, 
and by the H\"older inequality,
we see
\begin{align*}
\int_{M}\, \vert \psi \vert^{p} \,d\,m_{f} &\leq p\,C_{\delta}\, \int_{M}\, \vert \psi \vert^{p-1} \,  \Vert \nabla \psi\Vert \,d\,m_{f}\\
                                                           &\leq p\, C_{\delta}\, \left(\int_{M}\, \vert \psi \vert^{p}\,d\,m_{f}\right)^{1/q} \left(\int_{M}\, \Vert \nabla \psi\Vert^{p}\,d\, m_{f}\right)^{1/p}.
\end{align*}
Considering the Rayleigh quotient $R_{f,p}(\psi)$,
we complete the proof.
\end{proof}

Finally, we prove the following spectrum rigidity theorem:
\begin{thm}\label{thm:spectrum rigidity}
Let $p\in (1,\infty)$.
Let $\kappa<0$ and $\lambda:=\sqrt{\vert \kappa \vert}$.
For $N \in (-\infty,1]$,
assume that
$(M,\bm,f)$ has lower $(\kappa,\lambda,N)$-weighted curvature bounds.
Let $\bm$ be compact.
Suppose additionally that
there exists $\delta \in \mathbb{R}$ such that $f\leq (n-1)\delta$ on $M$.
Then
\begin{equation}\label{eq:spectrum rigidity}
\nu_{f,p}(M)\geq e^{-2p\delta}\left(\frac{(n-1)\lambda }{p}\right)^{p}.
\end{equation}
If the equality in $(\ref{eq:spectrum rigidity})$ holds,
then $M$ is isometric to $[0,\infty)\times_{F_{\kappa,\lambda}} \bm$;
moreover,
if $N\in (-\infty,1)$,
then $f\circ \gamma_{z}$ is constant for every $z\in \bm$.
\end{thm}
\begin{proof}
For $D>0$,
we see $C(\kappa,\lambda,D)=\left((n-1)\lambda \right)^{-1}\,\left(1-e^{-(n-1)\lambda\, D} \right)$.
Note that
the right hand side is monotone increasing as $D\to \infty$.
Put $D_{\delta}:=e^{-2\delta}\,\IR M$.
From Lemma \ref{lem:p-Laplacian1}
we conclude
\begin{equation*}
\nu_{f,p}(M)\geq e^{-2p\delta}\,(\,p\,C(\kappa,\lambda,D_{\delta})\,)^{-p} \geq e^{-2p\delta}\,\left(\,\frac{(n-1)\lambda}{p}\,\right)^{p}.
\end{equation*}

Assume
that the equality in (\ref{eq:spectrum rigidity}) holds.
Then the monotonicity of $C(\kappa,\lambda,D)$ with respect to $D$ implies $D_{\delta}=\infty$;
in particular,
we have $\IR M=\infty$.
Since $\bm$ is compact,
$\tau(z_{0})=\infty$ for some $z_{0}\in \bm$.
Theorem \ref{thm:splitting theorem} completes the proof of Theorem \ref{thm:spectrum rigidity}.
\end{proof}

\end{document}